\documentclass[12pt, onesided]{amsart}

\usepackage{amsmath}
\usepackage{amssymb}
\usepackage{amsthm}
\usepackage[all,cmtip]{xy}

\theoremstyle{plain}
\newtheorem{thm}{Theorem}[section]
\newtheorem{prop}[thm]{Proposition}
\newtheorem{lem}[thm]{Lemma}
\newtheorem{cor}[thm]{Corollary}
\newtheorem{conj}[thm]{Conjecture}

\newtheorem{fact}[thm]{Fact}

\theoremstyle{definition}
\newtheorem{defn}[thm]{Definition}
\newtheorem{eg}[thm]{Example}

\theoremstyle{remark}
\newtheorem{rem}[thm]{Remark}
\newtheorem{claim}[thm]{Claim}

\DeclareMathOperator{\Cl}{Cl}
\DeclareMathOperator{\codim}{codim}
\DeclareMathOperator{\Def}{Def}

\DeclareMathOperator{\Exc }{Exc}
\DeclareMathOperator{\Ext}{Ext}
\DeclareMathOperator{\Hom}{Hom}
\DeclareMathOperator{\Pic}{Pic}
\DeclareMathOperator{\Sing}{Sing}
\DeclareMathOperator{\Spec}{Spec}
\DeclareMathOperator{\depth}{depth}

\DeclareMathOperator{\Ker}{Ker}
\DeclareMathOperator{\Image}{Im}
\DeclareMathOperator{\Gal}{Gal}

\begin{document}

\title
[Deformations of $\mathbb{Q}$-Fano 3-folds ]
{On deformations of $\mathbb{Q}$-Fano threefolds}
\subjclass[2010]{Primary 14D15, 14J45; Secondary 14B07}
\keywords{deformation theory, $\mathbb{Q}$-Fano 3-folds, $\mathbb{Q}$-smoothings, general elephants}

\author{Taro Sano}
\address{Mathematics Institute, Zeeman Building, University of Warwick, Coventry, CV4 7AL, UK}
\email{T.Sano@warwick.ac.uk}

\address{Max Planck Institute for Mathematics, Vivatsgasse 7, 53111 Bonn, Germany}
\email{tarosano222@gmail.com}

\dedicatory{Dedicated to Professor~Yujiro~Kawamata on the~occasion of his~60th~birthday.}

\maketitle

\begin{abstract}
We study the deformation theory of a $\mathbb{Q}$-Fano 3-fold with only terminal singularities. First, we show that 
the Kuranishi space of a $\mathbb{Q}$-Fano 3-fold is smooth. Second, we show that 
 every $\mathbb{Q}$-Fano 3-fold with only ``ordinary'' terminal singularities is $\mathbb{Q}$-smoothable, 
 that is, it can be deformed to a $\mathbb{Q}$-Fano 3-fold with only quotient singularities. 
Finally, we prove $\mathbb{Q}$-smoothability of a $\mathbb{Q}$-Fano 3-fold  assuming the existence of a Du Val anticanonical element. 
As an application, we get the genus bound for primary $\mathbb{Q}$-Fano 3-folds with Du Val anticanonical elements. 
\end{abstract}

\tableofcontents

 \section{Introduction}
All algebraic varieties in this paper are defined over $\mathbb{C}$.

\subsection{Background and our results}
\begin{defn}
Let $X$ be a normal projective variety. 
We say that $X$ is {\it a $\mathbb{Q}$-Fano 3-fold} if $\dim X =3$, $X$ has only terminal singularities and $-K_X$ is an ample $\mathbb{Q}$-Cartier divisor. 
\end{defn}

$\mathbb{Q}$-Fano $3$-folds are important objects in the classification of algebraic varieties. 
Toward the classification of $\mathbb{Q}$-Fano 3-folds, it is fundamental to study their deformations. 
 
\begin{defn}
Let $X$ be an algebraic variety and $\Delta^1$ an open unit disc of dimension $1$. A {\it $\mathbb{Q}$-smoothing} of $X$ is 
a flat morphism of complex analytic spaces $f \colon \mathcal{X} \rightarrow \Delta^1$ such that $f^{-1}(0) \simeq X$ and $f^{-1}(t)$ has only quotient singularities of 
codimension at least $3$. 

If $X$ is proper, we assume that $f$ is a proper morphism. 
\end{defn}

\begin{rem}
Schlessinger \cite{Schl} proved that an isolated quotient singularity of dimension $\ge 3$ is infinitesimally rigid under small deformations. 
\end{rem}

Reid (\cite{pagoda}, \cite{YPG}) and Mori \cite{mori} showed that 
a $3$-fold terminal singularity can be written as a quotient of 
an isolated cDV hypersurface singularity by a finite cyclic group action and it admits a $\mathbb{Q}$-smoothing. 

In general, a local deformation may not lift to a global deformation. 
However, Alt{\i}nok--Brown--Reid conjectured the following in \cite[4.8.3]{ABR}. 

\begin{conj}\label{qsmqfanoconj}
 Let $X$ be a $\mathbb{Q}$-Fano $3$-fold. Then $X$ has a $\mathbb{Q}$-smoothing.   
 \end{conj}
 
The following theorem is an answer to their question in the ``ordinary'' case. 

\begin{thm}\label{qsmqfanointro}{\rm (= Corollary \ref{qsmordqfano})}
Let $X$ be a $\mathbb{Q}$-Fano $3$-fold with only ordinary terminal singularities (See Remark \ref{ordrem}). Then $X$ has a $\mathbb{Q}$-smoothing. 
\end{thm}

We prove a more general statement in Theorem \ref{qsmqfano} that implies Theorem \ref{qsmqfanointro}.  
 
\begin{rem}\label{ordrem}
 A 3-dimensional terminal singularity is called {\it ordinary} if the defining equation of its index 1 cover 
is $\mathbb{Z}_r$-invariant, where 
$\mathbb{Z}_r$ is the Galois group of the cover. In the list of 3-dimensional terminal singularities, 
there are 5 families of ordinary singularities and 1 {\it exceptional}  
 family of Gorenstein index 4 (See \cite[Theorem 12 (2)]{mori}  or \cite[(6.1) Figure (2)]{YPG} ). 
\end{rem}

Previously, Namikawa \cite{NamFano} proved that a Fano 3-fold with only terminal Gorenstein singularities admits a smoothing, that is, 
it can be deformed to a smooth Fano 3-fold. 
Minagawa \cite{mina} proved $\mathbb{Q}$-smoothability of a $\mathbb{Q}$-Fano 3-fold of Fano index one, that is, it has a global index one cover. 
Takagi also treated some cases in \cite[Theorem 2.1]{takagi2}. Note that the singularities on a $\mathbb{Q}$-Fano 3-fold in their cases are ordinary.  

In order to prove the $\mathbb{Q}$-smoothablity, we need the following theorem on the unobstructedness of deformations of a  $\mathbb{Q}$-Fano 3-fold. 

\begin{thm}\label{qfanounobsintro}{\rm (= Theorem \ref{unobsqfano})}
Let $X$ be a $\mathbb{Q}$-Fano $3$-fold. Then the deformations of $X$ are unobstructed. 
\end{thm}

Namikawa \cite{NamFano} proved the unobstructedness in the Gorenstein case and Minagawa \cite{mina} proved it 
for a $\mathbb{Q}$-Fano 3-fold of Fano index one. 
We show it for any $\mathbb{Q}$-Fano 3-fold. 
This theorem reduces the problem of finding good deformations to that of 1st order infinitesimal deformations. 
 
\vspace{5mm} 
 
Another fundamental problem in the classification of $\mathbb{Q}$-Fano 3-folds is to find anticanonical elements with only mild singularities.    
An anticanonical element is called an {\it elephant}. 
A Gorenstein Fano 3-fold with only canonical singularities has an elephant with only Du Val singularities(\cite{Shokurov}, \cite{ReidKawamata}). 
By using this fact, Mukai classified ``indecomposable'' Gorenstein Fano 3-folds with canonical singularities in \cite{mukai}. 
Hence the existence of a Du Val elephant is useful in the classification.  
However a $\mathbb{Q}$-Fano 3-fold may not have such a good element in general. 
There exist examples of $\mathbb{Q}$-Fano 3-folds with empty anticanonical linear systems 
or with only non Du Val elephants as in \cite[4.8.3]{ABR}. 
Nevertheless, Alt{\i}nok--Brown--Reid \cite{ABR} conjectured the following. 

\begin{conj}\label{simulqsmconj}
Let $X$ be a $\mathbb{Q}$-Fano $3$-fold. Assume that $|{-} K_X|$ contains an element $D$. 
\begin{enumerate}
\item Then there exists a deformation $f \colon \mathcal{X} \rightarrow \Delta^1$ of $X$ such that
 $|{-} K_{\mathcal{X}_t}|$ contains an element $D_t$ with only Du Val singularities for general $t \in \Delta^1$. 
\item Moreover, a divisor $D_t \subset \mathcal{X}_t$ is locally isomorphic to $\frac{1}{r}(a,r-a) \subset \frac{1}{r}(1,a,r-a)$, where both sides are corresponding 
cyclic quotient singularities for some coprime integers $r$ and $a$ 
around  each Du Val singularities of $D_t$.  
\end{enumerate}
\end{conj}

We call a deformation as above a {\it simultaneous $\mathbb{Q}$-smoothing} of a pair $(X,D)$. 
If we first assume the existence of a Du Val elephant, we get the following result, which is proved in Section \ref{simulqsmproofsection}. 

\begin{thm}\label{simulqsmqfanointro}
Let $X$ be a $\mathbb{Q}$-Fano $3$-fold. 
Assume that $|{-}K_X|$ contains an element $D$ with only Du Val singularities. 

Then $X$ has a simultaneous $\mathbb{Q}$-smoothing. In particular, $X$ has a $\mathbb{Q}$-smoothing.  
\end{thm} 
Note that we do not need the assumption of ordinary singularities as in Theorem \ref{qsmqfanointro}. 
The motivation of Conjecture \ref{simulqsmconj} is to treat a $\mathbb{Q}$-Fano $3$-fold with only non Du Val elephants. 
We investigate this case in elsewhere. 

\vspace{5mm}

A $\mathbb{Q}$-Fano 3-fold is called {\it primary} if its canonical divisor generates the class group mod torsion elements. 
Takagi \cite{Takdv} studied primary $\mathbb{Q}$-Fano 3-folds with only terminal quotient singularities and established the genus bound for those with 
Du Val elephants. Hence Theorems \ref{qsmqfanointro} and \ref{simulqsmqfanointro} are useful for the classification. 
Actually, as an application of Theorem \ref{simulqsmqfanointro}, we can reprove his bound as follows. 

\begin{cor}\label{genusboundintro}
Let $X$ be a primary $\mathbb{Q}$-Fano $3$-fold. 
Assume that $X$ is non-Gorenstein and $|{-}K_X|$ contains an element with only Du Val singularities.  

Then $h^0(X, {-}K_X) \le 10$. 
\end{cor}
Takagi expected the existence of a Du Val elephant for $X$ such that $h^0(X,-K_X)$ is appropriately big (\cite[p.37]{Takdv}). 
If we assume the expectation, Corollary \ref{genusboundintro} implies the genus bound as above for every primary $\mathbb{Q}$-Fano $3$-fold.

\vspace{5mm}

\subsection{Outline of the proofs}

We sketch the proof of the above theorems on a $\mathbb{Q}$-Fano 3-fold $X$.
 
First, we explain how to prove the unobstructedness briefly. 
If $X$ is Gorenstein, we have  
\[
\Ext^2_{\mathcal{O}_X}( \Omega^1_X, \mathcal{O}_X) \simeq \Ext^2_{\mathcal{O}_X}( \Omega^1_X \otimes \omega_X, \omega_X) \simeq H^1(X, \Omega^1_X \otimes \omega_X)^* 
\]
since $\omega_X$ is invertible and the unobstructedness is reduced to the Kodaira-Nakano type vanishing of the cohomology. 
However, if $X$ is non-Gorenstein, that is, $\omega_X$ is not invertible, we can not reduce the vanishing of the $\Ext$ group to the vanishing of cohomology groups a priori and 
we do not have a direct method to prove the vanishing of the $\Ext$ group. 
Moreover, since we do not have a branched cover of a $\mathbb{Q}$-Fano 3-fold which is Fano or Calabi-Yau in the general case, 
we can not reduce the unobstructedness to that of such cover.  
We solve this difficulty by considering the obstruction classes rather than the ambient obstruction space $\Ext^2$ and considering the smooth part. 
The important point is that  deformations of $X$ are bijective to deformations of the smooth part as in \cite[12.1.8]{kollarmoriflip} or 
\cite[Theorem 12]{kollarflat}. 
The description of the obstruction by a 2-term extension as in Proposition \ref{obs} is a crucial tool.

In order to find a good deformation of first order, we follow the line of the proof in the case of Fano index 1 by Minagawa \cite{mina} 
which used  \cite[Theorem 1]{NamSt} of Namikawa-Steenbrink  on the non-vanishing of the homomorphism between cohomology groups. 
We need a generalisation of this theorem to the non-Gorenstein setting which is Proposition \ref{ordinarynonzero}. 
We can generalise this lemma provided that the singularity is ordinary. 
 The generalisation of this lemma for general terminal singularities implies Conjecture \ref{qsmqfanoconj}.

Now, in order to find a good deformation of first order under the assumption of a Du Val elephant, 
we use the deformation theory of the pair of $X$ and $D$ where $D \in |{-} K_X|$. 
The smoothness of the Kuranishi space of $X$ implies that the smoothness of the Kuranishi space of the pair $(X,D)$ for 
$D \in |{-}K_X|$ (Theorem \ref{pairunobs}). 
The important point in the proof is that an elephant contains the non-Gorenstein points of $X$. 
By this, in order to see that a deformation of $X$ is a $\mathbb{Q}$-smoothing, it is 
enough to see that the singularities of $D$ deforms non trivially.  
Here we adapt the diagram of \cite[Theorem 1.3]{NamSt} to the case $(X,D)$. 
Instead of the Namikawa-Steenbrink's proposition \cite[Theorem 1.1]{NamSt}  on non-vanishing of a certain cohomology map, we use 
the coboundary map of the local cohomology sequence for the pair.  
To use such a map, we arrange a resolution of singularities of the pair which has 
non-positive discrepancies as in Proposition \ref{essresolprop}. 
Moreover we refine the Lefschetz theorem for class groups by Ravindra-Srinivas \cite{RSGL} for our cases (Proposition \ref{lefprop}) and this Lefschetz statement plays 
an important role for lifting.

\section{Unobstructedness of deformations of a $\mathbb{Q}$-Fano 3-fold}

\subsection{Preliminaries on infinitesimal deformations} 
First, we introduce a deformation functor of an algebraic scheme. 
\begin{defn}\label{defschdefn}(cf. \cite[1.2.1]{Sernesi}) 
Let $X$ be an algebraic scheme over $k$ and $S$ an algebraic scheme over $k$ with 
a closed point $s \in S$. 
A {\it deformation} of $X$ over $S$ is a pair $(\mathcal{X}, i)$, 
where $\mathcal{X}$ is a scheme flat over $S$ and 
$i \colon X \hookrightarrow \mathcal{X}$ is a closed immersion such that 
the induced morphism $X \rightarrow \mathcal{X} \times_S \{s \}$ is an isomorphism. 

Two deformations $(\mathcal{X}_1, i_1)$ and $(\mathcal{X}_2, i_2)$ over $S$ are said to be 
{\it equivalent} if there exists an isomorphism $\varphi \colon \mathcal{X}_1 \rightarrow \mathcal{X}_2$ over $S$ 
which commutes the following diagram; 
\[
\xymatrix{
X \ar@{^{(}->}^{i_1}[r] \ar@{^{(}->}^{i_2}[rd] & \mathcal{X}_1 \ar[d]^{\varphi} \\
 & \mathcal{X}_2
}
\]
Let $\mathcal{A}$ be the category of Artin local $k$-algebras with residue field $k$. 
We define the functor $\Def_X \colon \mathcal{A} \rightarrow (Sets)$ by setting 
\begin{equation}\label{functordef}
\Def_X(A):= \{ (\mathcal{X}, i): \text{deformation of } X \text{ over } \Spec A \}/({\rm equiv}),    
\end{equation}
where $({\rm equiv})$ means the equivalence introduced in the above. 
\end{defn} 

We also introduce the deformation functor of a closed immersion. 

\begin{defn}\label{pairdefschdefn}(cf. \cite[3.4.1]{Sernesi})
Let $f \colon D \hookrightarrow X$ be a closed immersion of algebraic schemes over an algebraically closed field $k$ and $S$ an algebraic scheme over $k$ with 
a closed point $s \in S$. 
A {\it deformation} of a pair $(X,D)$ over $S$ is a data $(F, i_X, i_D)$ in the cartesian diagram 
\begin{equation}\label{deformmapdiagram}
\xymatrix{
D \ar@{^{(}->}[r]^{i_D} \ar[d]^{f} & \mathcal{D} \ar[d]^{F} \\
X \ar@{^{(}->}[r]^{i_X}  \ar[d] & \mathcal{X} \ar[d]^{\Psi} \\ 
\{s \} \ar@{^{(}->}[r] & S, 
}
\end{equation}
where $\Psi$ and $\Psi \circ F$ are flat and $i_D, i_X$ are closed immersions. 
Two deformations $(F, i_D, i_X)$ and $(F', i_D', i_X')$ of $(X,D)$ over $S$ are said to be 
{\it equivalent} if there exist isomorphisms $\alpha \colon \mathcal{X} \rightarrow \mathcal{X}'$ and 
$\beta \colon \mathcal{D} \rightarrow \mathcal{D}'$ over $S$ 
which commutes the following diagram; 
\[
\xymatrix{
D \ar@{^{(}->}[r]^{i_D} \ar@{^{(}->}[rd]^{i'_{D}} & \mathcal{D} \ar[r] \ar[d]^{\beta} & \mathcal{X} \ar[d]^{\alpha} & 
X \ar@{^{(}->}[l]^{i_X} \ar@{^{(}->}[ld]^{i_X'} \\
 & \mathcal{D}' \ar[r] & \mathcal{X}'. &   
}
\] 
 
We define the functor $\Def_{(X,D)} \colon \mathcal{A} \rightarrow (Sets)$ by setting 
\begin{equation}\label{pairfunctordef}
\Def_{(X,D)}(A):= \{ (F, i_D, i_X): \text{deformation of } (X,D) \text{ over } \Spec A \}/({\rm equiv}),    
\end{equation}
where $({\rm equiv})$ means the equivalence introduced in the above. 
\end{defn}

We study unobstructedness of the above functors in this section. 
Unobstructedness is defined as follows. 

\begin{defn}\label{unobsfunc}
We say that deformations of $X$ are {\it unobstructed} if, for all $A, A' \in \mathcal{A}$ with an exact sequence 
\[
0 \rightarrow J \rightarrow A' \rightarrow A \rightarrow 0
\]
such that $\mathfrak{m}_{A'} \cdot J =0$, the natural restriction map of deformations 
\[
\Def_X (A') \rightarrow \Def_X (A) 
\] 
is surjective, that is, $\Def_X$ is a smooth functor. 
\end{defn}

\begin{prop}\label{enough}
Let $X$ be an algebraic scheme with a versal formal couple $(R, \hat{u})$ in the sense of  \cite[Definition 2.2.6]{Sernesi}. 
Set $A_m:= k[t]/(t^{m+1})$ for all integers $m \ge 0$. 
Assume that 
\[
\Def_X(A_{n+1}) \rightarrow \Def_X(A_n) 
\]
are surjective for all non-negative integers $n \ge 0$. 

Then deformations of $X$ are unobstructed. 
\end{prop}

\begin{proof}
For $A \in \mathcal{A}$, let $h_R (A)$  be the set of local $k$-algebra homomorphisms from $R$ to $A$. This rule defines a functor
\[
h_R \colon \mathcal{A} \rightarrow ({\rm Sets}).  
\]   

Since $(R,\hat{u})$ is versal, we have a smooth morphism of functors 
\[
\phi_{\hat{u}} \colon h_R \rightarrow \Def_X 
\]
defined by $\hat{u}$. 

Then we can see that 
\[
h_R( A_{n+1}) \rightarrow h_R(A_n)
\]
are surjective for all $n$ by the assumption and the versality. 

By  \cite[Lemma 5.6]{FM} and the assumption, we can see that $h_R$ is a smooth functor. 
This implies that $\Def_X$ is smooth. 
\end{proof}

We use the following lemma about an isomorphism of some $\Ext$ groups. 

\begin{lem}\label{extisom}
Let $X$ be an algebraic scheme over an algebraically closed field $k$. 
Let $\mathcal{X} \in \Def_X(A)$ be a deformation of $X$ over $A \in \mathcal{A}$. 
Let $\mathcal{F}$ be a coherent  $\mathcal{O}_{\mathcal{X}}$-module which is flat over $A$. 
Let $\mathcal{G}$ be a coherent $\mathcal{O}_X$-module which is also an $\mathcal{O}_{\mathcal{X}}$-module by the canonical surjection 
$\mathcal{O}_{\mathcal{X}} \twoheadrightarrow \mathcal{O}_X$ . 
Then we have the following; 
\begin{enumerate}
\item[(i)] $ \underline{\Ext}^i_{\mathcal{O}_{\mathcal{X}}}(\mathcal{F}, \mathcal{G}) \simeq \underline{\Ext}^i_{\mathcal{O}_{X}}(\mathcal{F} \otimes_A k, \mathcal{G})$ for all $i$, 
 where $\underline{\Ext}^i$ is a sheaf of $\Ext$ groups. 
\item[(ii)]  $\Ext^i_{\mathcal{O}_{\mathcal{X}}}(\mathcal{F}, \mathcal{G}) \simeq \Ext^i_{\mathcal{O}_{X}}(\mathcal{F} \otimes_A k, \mathcal{G})$ for all $i$. 
\end{enumerate}
\end{lem}

\begin{proof}
(i) Let $\mathcal{E}_{\bullet} \rightarrow \mathcal{F} \rightarrow 0$ be a resolution of $\mathcal{F}$ by a complex $\mathcal{E}_{\bullet}$ 
of locally free $\mathcal{O}_{\mathcal{X}}$-modules.  
By \cite[Proposition 6.5]{HartAG}, we see that 
\begin{equation}\label{1stisomcalc}
\mathcal{H}^i (\underline{\Hom}_{\mathcal{O}_{\mathcal{X}}}(\mathcal{E}_{\bullet}, \mathcal{G})) \simeq \underline{\Ext}^i_{\mathcal{O}_{\mathcal{X}}}(\mathcal{F}, \mathcal{G}), 
\end{equation}
where $\mathcal{H}^i$ is a cohomology sheaf and $\underline{\Hom}$ is a sheaf of $\Hom$ groups. 
Since $\mathcal{F}$ is flat over $A$, we see that $\mathcal{E}_{\bullet} \otimes_A k \rightarrow \mathcal{F} \otimes_A k \rightarrow 0$ is still 
a resolution of the sheaf $\mathcal{F} \otimes_A k$. 
Hence we have 
\begin{equation}\label{2ndisomcalc}
\mathcal{H}^i ( \underline{\Hom}_{\mathcal{O}_{X}}(\mathcal{E}_{\bullet} \otimes_A k, \mathcal{G})) \simeq 
\underline{\Ext}^i_{\mathcal{O}_X} (\mathcal{F} \otimes_A k, \mathcal{G}).  
\end{equation}
Note that $\underline{\Hom}_{\mathcal{O}_{\mathcal{X}}}(\mathcal{E}_{\bullet}, \mathcal{G}) \simeq 
\underline{\Hom}_{\mathcal{O}_{X}}(\mathcal{E}_{\bullet} \otimes_A k, \mathcal{G})$ since $\mathcal{G}$ is an $\mathcal{O}_X$-module. 
By this and isomorphisms (\ref{1stisomcalc}) and (\ref{2ndisomcalc}), we obtain the required isomorphism in (i). 

(ii) This follows from (i) and the local-to-global spectral sequence of $\Ext$ groups; 
\[
H^i(\mathcal{X}, \underline{\Ext}_{\mathcal{O}_{\mathcal{X}}}^j(\mathcal{F}, \mathcal{G})) \Rightarrow \Ext_{\mathcal{O}_{\mathcal{X}}}^{i+j} (\mathcal{F}, \mathcal{G}) 
\]
 \end{proof}


\subsection{Description of obstruction classes}

We need the following description of the obstruction space for deformations.

\begin{prop}\label{obs}
Let $k$ be an algebraically closed field of characteristic $0$. 
Let $X$ be a reduced scheme of finite type over $k$.  
Let $U \subset X$ be an open subset with only l.c.i. singularities and $\iota \colon U \rightarrow X$  an inclusion map. 
Assume that $\depth \mathcal{O}_{X, \mathfrak{p}}  \ge 3$ for all scheme theoretic points $\mathfrak{p} \in X \setminus U$. 
(We obtain $\codim_X X \setminus U \ge 3$ by this condition.)
Let $\Omega^1_U$ be the K\"{a}hler differential sheaf on $U$. Set $A_n:= k[t]/(t^{n+1})$ and let 
\[
\xi_n:= (f_n \colon \mathcal{X}_n \rightarrow \Spec A_n) 
\]
 be a deformation of $X$.  
 
Then the obstruction to lift $\mathcal{X}_n$ over $A_{n+1}$ lies in $\Ext^2_{\mathcal{O}_U}( \Omega^1_U, \mathcal{O}_U )$.  
\end{prop}

\begin{proof}

We need to define an element 
\[
o_{\xi_n} \in \Ext^2_{\mathcal{O}_U} (\Omega^1_U, \mathcal{O}_U)
\]
which has a property that $o_{\xi_n} = 0$ if and only if there is a deformation  
$\xi_{n+1} = (f_{n+1} \colon \mathcal{X}_{n+1} \rightarrow \Spec A_{n+1})$ which sits in the following cartesian diagram; 
\begin{equation}\label{liftdiag}
\xymatrix{
      \mathcal{X}_{n+1} \ar[d] &  \mathcal{X}_n \ar[l] \ar[d] \\
     \Spec A_{n+1} & \Spec A_n . \ar[l]}  
 \end{equation}

Since the characteristic of $k$ is zero, we have 
\[\Omega^1_{A_n/k} \simeq A_{n-1}
\]
 as $A_n$-modules and  
an exact sequence   
\begin{equation}\label{artindiffexact}
0 \rightarrow (t^{n+1}) \stackrel{d}{\rightarrow} \Omega^1_{A_{n+1}/k} \otimes_{A_{n+1}} A_n 
\rightarrow \Omega^1_{A_n/k} \rightarrow 0. 
\end{equation}

Let $f_{\mathcal{U}_n} \colon \mathcal{U}_n \rightarrow \Spec A_n$ be the flat deformation of $U$ induced by $f_n$. 
By pulling back the above sequence by the flat morphism $f_{\mathcal{U}_n}$, we get the following exact sequence; 

\begin{equation}\label{conormal}
0 \rightarrow \mathcal{O}_U \rightarrow f_{\mathcal{U}_n}^{\ast} (\Omega^1_{\Spec A_{n+1}/k} |_{\Spec A_n}) \rightarrow 
f_{\mathcal{U}_n}^{\ast} \Omega^1_{\Spec A_n/k} \rightarrow 0. 
\end{equation}

Then, there is the relative cotagent sequence of a relative l.c.i. morphism $f_{\mathcal{U}_n}$ (cf. \cite[Theorem D.2.8]{Sernesi}); 

\begin{equation}\label{cot}
0 \rightarrow f_{\mathcal{U}_n}^{\ast} \Omega^1_{\Spec A_n / k} \rightarrow \Omega^1_{\mathcal{U}_n/k} \rightarrow 
\Omega^1_{\mathcal{U}_n/ \Spec A_n} \rightarrow 0. 
\end{equation}

By combining the sequences (\ref{conormal}), (\ref{cot}), we get the following exact sequence;  

\begin{equation}\label{obsseq}
0 \rightarrow \mathcal{O}_U \rightarrow f_{\mathcal{U}_n}^{\ast} (\Omega^1_{\Spec A_{n+1}/k} |_{\Spec A_n}) 
\rightarrow \Omega^1_{\mathcal{U}_n/k} \rightarrow 
 \Omega^1_{\mathcal{U}_n/ \Spec A_n} \rightarrow 0. 
\end{equation}

Let \[
o_{\xi_n} \in \Ext^2_{\mathcal{O}_U} ( \Omega^1_U, \mathcal{O}_U) 
\simeq \Ext^2_{\mathcal{O}_{\mathcal{U}_n}} (  \Omega^1_{\mathcal{U}_n/ \Spec A_n}, \mathcal{O}_U)
\]
 be the element corresponding to the 
exact sequence (\ref{obsseq}). Note that $\Omega^1_{\mathcal{U}_n/ \Spec A_n}$ is a flat $A_n$-module since $U$ is generically smooth and 
has only l.c.i.\ singularities (\cite[Theorem D.2.7]{Sernesi}). 
Hence we obtain the above isomorphism of $\Ext^2$ by applying Lemma \ref{extisom}. 

We check that this $o_{\xi_n}$ is the obstruction to the existence of lifting of $\xi_n$ over $A_{n+1}$. 
 
Suppose that we have a lifting $\xi_{n+1}= (f_{n+1} \colon \mathcal{X}_{n+1} \rightarrow \Spec A_{n+1}) $ 
with the diagram (\ref{liftdiag}). 
Then we can see that $o_{\xi_n} =0$ as in  \cite[Proposition 2.4.8]{Sernesi}. 
 
Conversely, suppose that $o_{\xi_n}=0$. Consider the following exact sequence  
\[
\Ext^1_{\mathcal{O}_{\mathcal{U}_n}} (  \Omega^1_{\mathcal{U}_n/k}, \mathcal{O}_U) \stackrel{\epsilon}{\rightarrow} 
\Ext^1_{\mathcal{O}_{\mathcal{U}_n}} (f_{\mathcal{U}_n}^{\ast}\Omega^1_{\Spec A_n /k}, \mathcal{O}_U)
\stackrel{\delta}{\rightarrow} 
\Ext^2_{\mathcal{O}_{\mathcal{U}_n}} ( \Omega^1_{\mathcal{U}_n/\Spec A_n}, \mathcal{O}_U ) 
\] 
which is induced by the exact sequence (\ref{cot}). 
Consider 
\[
\gamma \in  \Ext^1_{\mathcal{O}_{\mathcal{U}_n}} (f_{\mathcal{U}_n}^{\ast} \Omega^1_{\Spec A_n /k}, \mathcal{O}_U)
\]
 which corresponds to the exact sequence (\ref{conormal}). 
It is easy to see that $\delta (\gamma) = o_{\xi_n}$.  
Hence there exists $\gamma' \in \Ext^1_{\mathcal{O}_{\mathcal{U}_n}} (  \Omega^1_{\mathcal{U}_n/k}, \mathcal{O}_U)$ such that 
$\epsilon (\gamma')=\gamma $.  
The class $\gamma'$ corresponds to the following short exact sequence 
\[
0 \rightarrow \mathcal{O}_U \rightarrow \mathcal{E} \rightarrow   \Omega^1_{\mathcal{U}_n/k} \rightarrow 0 
\]
for some $\mathcal{O}_{\mathcal{U}_n}$-module $\mathcal{E}$ on $\mathcal{U}_n$. 

We can construct a sheaf of rings 
$\mathcal{O}_{\mathcal{U}_{n+1}}$ as the fiber product  
\[
\mathcal{O}_{\mathcal{U}_{n+1}} := \mathcal{E} \times_{ \Omega^1_{\mathcal{U}_n/k}} \mathcal{O}_{\mathcal{U}_n}
\]
 as in 
 \cite[Theorem 1.1.10]{Sernesi}. We can define a multiplication of $(\xi, f), (\xi', f') \in \mathcal{O}_{\mathcal{U}_{n+1}}$ by 
 \[
 (\xi, f) \cdot (\xi', f'):= (f'\xi+f \xi' , ff'). 
 \] We also have a commutative diagram 
\begin{equation}\label{algextu}
	\xymatrix{
0 \ar[r] & t^{n+1} \cdot \mathcal{O}_U \ar[r] \ar[d]^{=} & \mathcal{O}_{\mathcal{U}_{n+1}} \ar[r] \ar[d] & 
 \mathcal{O}_{\mathcal{U}_n} \ar[r] \ar[d]^{d} & 0 \\
 0 \ar[r] & t^{n+1} \cdot \mathcal{O}_U \ar[r] & \mathcal{E} \ar[r] & 
  \Omega^1_{\mathcal{U}_n/k} \ar[r] & 0,  
}
\end{equation}
where the upper horizontal sequence is an exact sequence of sheaves of rings. 
We can put an $A_{n+1}$-algebra structure on $\mathcal{O}_{\mathcal{U}_{n+1}}$ as follows (cf.\ \cite[p.10]{namikawaronsetsu}); 

We have a commutative diagram 
\[
\xymatrix{
0 \ar[r] & t^{n+1} \cdot \mathcal{O}_{U} \ar[r] \ar[d]^{=} & f_{\mathcal{U}_n}^*(\Omega^1_{A_{n+1}/k} \otimes_{A_{n+1}} A_n) \ar[r] \ar[d]^{g_{\mathcal{E}}} & 
f_{\mathcal{U}_n}^*(\Omega^1_{A_n/k}) \ar[r] \ar[d] & 0 \\  
0 \ar[r] & t^{n+1} \cdot \mathcal{O}_{U} \ar[r]  & \mathcal{E} \ar[r]^{f_{\mathcal{E}}}  & 
\Omega^1_{\mathcal{U}_n/k} \ar[r]  & 0. 
}
\]
We have an element $f_{\mathcal{U}_n}^*(dt) \in H^0(\mathcal{U}_n, f_{\mathcal{U}_n}^*(\Omega^1_{A_{n+1}/k} \otimes_{A_{n+1}} A_n))$ and 
can define 
\[
t_{\mathcal{U}_{n+1}}:= (g_{\mathcal{E}} (f_{\mathcal{U}_n}^* (dt)), t) \in H^0(\mathcal{U}_{n+1},  \mathcal{O}_{\mathcal{U}_{n+1}}). 
\]
Since we can calculate $t_{\mathcal{U}_{n+1}}^{n+2}=0$, we can define a homomorphism 
$\varphi_{n+1} \colon A_{n+1} \rightarrow \mathcal{O}_{\mathcal{U}_{n+1}}$ 
such that $\varphi_{n+1}(t) = t_{\mathcal{U}_{n+1}}$ 
and put an $A_{n+1}$-algebra structure on $\mathcal{O}_{\mathcal{U}_{n+1}}$. 

We can check that $\mathcal{O}_{\mathcal{U}_{n+1}} \otimes_{A_{n+1}} A_n \simeq \mathcal{O}_{\mathcal{U}_n}$ and 
$\mathcal{O}_{\mathcal{U}_{n+1}} \otimes_{A_{n+1}} (t^{n+1}) \simeq (t^{n+1}) \mathcal{O}_X$. 
Thus, by the local criterion of flatness (\cite[Proposition 2.2]{HartDef}), we see that $\mathcal{O}_{\mathcal{U}_{n+1}}$ 
is flat over $A_{n+1}$.

We have the following claim. 

\begin{claim}\label{R1O} 
\begin{enumerate}
\item[(i)] $R^1 \iota_* \mathcal{O}_U = 0$.
\item[(ii)] Let $M$ be a finite $A_n$-module. Then 
\[
R^1 \iota_* (f_{\mathcal{U}_n}^* \widetilde{M}) =0,  
\]
where $\widetilde{M}$ is a 
coherent sheaf on $\Spec A_n$ associated to $M$. 
\end{enumerate}
\end{claim}

\begin{proof}[Proof of Claim]
\noindent(i) Let $p \in X \setminus U$ be a point and $U_p$ a small affine neighborhood of $p$. Put $Z_p:= U_p \cap (X \setminus U)$. 
It is enough to show that  $H^1(U_p \setminus Z_p, \mathcal{O}_{U_p \setminus Z_p}) =0$. 
We have $H^2_{Z_p}(U_p, \mathcal{O}_{U_p}) =0$ since $\depth_{q} \mathcal{O}_{X,q} \ge 3$ for all scheme-theoretic point 
$q \in Z_p$ by the hypothesis. Since $H^i(U_p, \mathcal{O}_{U_p}) = 0$ for $i=1,2$, we have $H^1(U_p \setminus Z_p, \mathcal{O}_{U_p}) \simeq 
H^2_{Z_p}(U_p, \mathcal{O}_{U_p})  =0$. 

\noindent(ii)
We proceed by induction on $\dim_k M$. 

If $M \simeq k$, then this is the first claim. 

Now assume that there is an exact sequence 
\[
0 \rightarrow k \rightarrow M \rightarrow M' \rightarrow 0
\] 
of $A_n$-modules and the claim holds for $M'$. 
Then we have an exact sequence 
\[
R^1 \iota_* (f_{\mathcal{U}_n}^* \widetilde{k}) \rightarrow R^1 \iota_* (f_{\mathcal{U}_n}^* \widetilde{M}) 
\rightarrow R^1 \iota_* (f_{\mathcal{U}_n}^* \widetilde{M'}) 
\]
and the left and right hand sides are zero by the induction hypothesis. Hence $R^1 \iota_* (f_{\mathcal{U}_n}^* \widetilde{M})=0$. 
\end{proof}

Note that $\iota_* \mathcal{O}_U \simeq \mathcal{O}_X, \iota_* \mathcal{O}_{\mathcal{U}_n} \simeq \mathcal{O}_{\mathcal{X}_n}$ 
by Claim \ref{R1O}. 
Set $\mathcal{O}_{\mathcal{X}_{n+1}} := \iota_* \mathcal{O}_{\mathcal{U}_{n+1}}$. 
By taking $\iota_*$ of (\ref{algextu}), we have an exact sequence 
\begin{equation}\label{algext}
0 \rightarrow t^{n+1} \cdot \mathcal{O}_X \rightarrow \mathcal{O}_{\mathcal{X}_{n+1}} \rightarrow \mathcal{O}_{\mathcal{X}_n} \rightarrow 0  
\end{equation}
since $R^1 \iota_* \mathcal{O}_U =0$. 
Thus we see that $\mathcal{O}_{\mathcal{X}_{n+1}} \otimes_{A_{n+1}} A_n \simeq \mathcal{O}_{\mathcal{X}_n}$. 
We can see that $\mathcal{O}_{\mathcal{X}_{n+1}}$ is flat over $A_{n+1}$ by \cite[Theorem 12]{kollarflat} since $\mathcal{O}_{\mathcal{U}_{n+1}}$ 
is flat over $A_{n+1}$. 

Let  $\mathcal{X}_{n+1}:=(X, \mathcal{O}_{\mathcal{X}_{n+1}})$ be the scheme defined by the sheaf $\mathcal{O}_{\mathcal{X}_{n+1}}$.  
Then the morphism $\mathcal{X}_{n+1} \rightarrow \Spec A_{n+1}$ is flat and  
\[
\xi_{n+1}:= ( \mathcal{X}_{n+1} \rightarrow \Spec A_{n+1}) 
\] is a lifting of $\xi_n$. 
\end{proof}

\begin{rem}
The author does not know whether the above construction of obstruction classes works for general $A,A'$ as in Definition \ref{unobsfunc}. 
Actually, the exact sequence (\ref{artindiffexact}) is not exact for a small extension $A'$ of $A$ in general. 
An example of such a small extension by Manetti is given in \cite{Sernesierrata}. 

However Proposition \ref{enough} reduces the study of unobstructedness to the case $A= A_n, A' = A_{n+1}$. 
\end{rem}

\subsection{Proof of Theorem \ref{qfanounobsintro}}

We need the following Lefschetz type theorem. 

\begin{thm}\label{lef}  \cite[Chapter 3.1. Theorem]{F} 
Let $X \subset \mathbb{P}^N$ be a projective variety of dimension $n$ and $L \subset \mathbb{P}^N$ a linear subspace of codimension $d \le n$. 
Assume that $X \setminus (X \cap L)$ has only l.c.i. singularities. Then the relative homotopy group satisfies 
\[
\pi_i (X, X \cap L) = 0  \ \ \ \ \ \ (i \leq n-d).  
\]
In particular,  the restriction map $H^i(X, \mathbb{C}) \rightarrow H^i(X\cap L, \mathbb{C})$ is injective for $i \leq n-d$. 
\end{thm} 

We also need the following lemma on flatness of some sheaf. 

\begin{lem}\label{iotaflatlem}
	Let $X$ be a $3$-fold with only terminal singularities 
	and $U$ its regular part with an open immersion $\iota \colon U \hookrightarrow X$. 
	Let $\xi_n:= (\mathcal{X}_n \rightarrow \Spec A_n)$ be a deformation of $X$ over $A_n= \mathbb{C}[t]/(t^{n+1})$ and 
	$\mathcal{U}_n \rightarrow \Spec A_n$ the deformation of $U$ induced by $\xi_n$. 
	
	Then the sheaf $\iota_*(\Omega^1_{\mathcal{U}_n/A_n} \otimes \omega_{\mathcal{U}_n/A_n})$ is flat over $A_n$. 
	Moreover, we have an isomorphism 
	\begin{equation}\label{basechangeisom}
	\left( \iota_* (\Omega^1_{\mathcal{U}_n/A_n} \otimes \omega_{\mathcal{U}_n/A_n}) \right) \otimes_{A_n} \mathbb{C} \simeq \iota_* (\Omega^1_U \otimes \omega_U)
	\end{equation}
	\end{lem}  

\begin{proof}
	Since we can check this locally, we can assume $X$ is a Stein neighborhood of a singularity $p \in X$. 
	
	Let $\omega^{[i]}_{\mathcal{X}_n/A_n} := \iota_* \omega_{\mathcal{U}_n/A_n}^{\otimes i}$ and $r$ the Gorenstein index of $X$. 
	We see that $\omega^{[i]}_{\mathcal{X}_n/A_n}$ is flat over $A_n$ by \cite[Theorem 12]{kollarflat} since $\mathcal{O}_X(iK_X)$ is ${\rm S}_3$ 
	(cf.\ \cite[Corollary 5.25]{KM}). 
	The isomorphism $\omega^{[r]}_{\mathcal{X}_n/A_n} \simeq \mathcal{O}_{\mathcal{X}_n}$ determines a cyclic cover  
	\[
	\pi_n \colon \mathcal{Y}_n := \Spec_{\mathcal{O}_{\mathcal{X}_n}} \oplus_{i=0}^{r-1} \omega_{\mathcal{X}_n/A_n}^{[i]} \rightarrow \mathcal{X}_n   
	\]
	and we see that $\mathcal{Y}_n$ is flat over $A_n$. 
	We also see that $\Omega^1_{\mathcal{Y}_n/A_n}$ is flat over $A_n$ since $\mathcal{Y}_n \rightarrow \Spec A_n$ is relative l.c.i.\ morphism and generically smooth 
	(cf.\ \cite[Theorem D.2.7]{Sernesi}).

	Let $\mathcal{Y}'_n:= \pi_n^{-1}(\mathcal{U}_n)$. 
	Next we see the isomorphism 
	\begin{equation}\label{iotaisom}
	\iota_* \Omega^1_{\mathcal{Y}'_n/A_n} \simeq \Omega^1_{\mathcal{Y}_n/A_n}.   
	\end{equation}
	We show this by induction on $n$. The isomorphism for $n=0$ is known. (cf.\ \cite{Kunz}, \cite[Theorem 1.2]{GrebRoll}) 
	Assume we have the isomorphism for $i \le n-1$. 
	We have a commutative diagram 
	\[
	\xymatrix{
	0 \ar[r] & \Omega^1_Y \ar[r] \ar[d]^{\simeq} & \Omega^1_{\mathcal{Y}_n/A_n} \ar[r] \ar[d] & 
	\Omega^1_{\mathcal{Y}_{n-1}/A_{n-1}} \ar[r] \ar[d]^{\simeq} & 0 \\
	o \ar[r] & \iota_* \Omega^1_U \ar[r] & \iota_* \Omega^1_{\mathcal{U}_n /A_n} \ar[r]^{r'_{n,n-1}} & 
	\iota_* \Omega^1_{\mathcal{U}_{n-1} /A_{n-1}}.   &   
	}
	\]
	We see that $r'_{n,n-1}$ is surjective by the above diagram.   
	Thus we see that the vertical homomorphism in the middle is also an isomorphism. 
	Thus we obtain the required isomorphism (\ref{iotaisom}). 
	
	By the isomorphism (\ref{iotaisom}), we obtain an isomorphism 
	\[
	\Omega^1_{\mathcal{Y}_n/A_n} \simeq \iota_* \Omega^1_{\mathcal{Y}'_n/A_n} 
	\simeq \oplus_{i=0}^{r-1} \iota_* (\Omega^1_{\mathcal{U}_n/A_n} \otimes \omega_{\mathcal{U}_n/A_n}^{\otimes i}) 
	\]
	since $\pi_n|_{\mathcal{Y}'_n} \colon \mathcal{Y}'_n \rightarrow \mathcal{U}_n$ is \'{e}tale. 
	Since $\Omega^1_{\mathcal{Y}_n/A_n}$ is flat over $A_n$ and the direct summand of a flat module is again flat, 
	we see that the sheaf $\iota_*(\Omega^1_{\mathcal{U}_n/A_n} \otimes \omega_{\mathcal{U}_n/A_n})$ is flat over $A_n$. 
	
	Now we check the isomorphism (\ref{basechangeisom}). We can also check this locally 
	and may assume that $X$ is Stein. We use the same notations as above. 
	Let $\pi \colon Y:= \mathcal{Y}_n \otimes_{A_n} \mathbb{C} \rightarrow X$ be the index one cover induced by $\pi_n$ and  
	$Y':= \pi^{-1}(U)  \subset Y$. We have isomorphisms
	\[
	\Omega^1_{\mathcal{Y}_n/A_n} \otimes_{A_n} \mathbb{C} \simeq \Omega^1_Y \simeq \iota_* \Omega^1_{Y'} \simeq 
	\oplus_{i=0}^{r-1} \iota_* (\Omega^1_{U} \otimes \omega_U^{\otimes i}), 
	\]
	\[
	\Omega^1_{\mathcal{Y}_n/A_n} \otimes_{A_n} \mathbb{C} \simeq 
	\oplus_{i=0}^{r-1} \left(\iota_* (\Omega^1_{\mathcal{U}_n/A_n} \otimes \omega_{\mathcal{U}_n/A_n}^{\otimes i})  \right) \otimes_{A_n} \mathbb{C}. 
	\]
	Comparing the $\mathbb{Z}_r$-eigenparts for $i=1$, we obtain the required isomorphism. 
	 
	Thus we finish the proof of Lemma \ref{iotaflatlem}. 
	\end{proof}

By using the obstruction class in Proposition \ref{obs}, we can show the following theorem. 

\begin{thm}\label{unobsqfano}
Let $X$ be a $\mathbb{Q}$-Fano $3$-fold. Then  
deformations of $X$ are unobstructed. 
\end{thm}

\begin{proof}
Let $U$ be the smooth part of $X$. Note that $\codim_X X \setminus U \ge 3$ and $X$ is Cohen-Macaulay    
since $X$ has only terminal singularities. Hence $X$ and $U$ satisfy the assumption of Proposition \ref{obs}. 
Set $k:= \mathbb{C}$. 

Let $\xi_n \in \Def_X (A_n)$ be a deformation of $X$ 
\[
f_n \colon \mathcal{X}_n \rightarrow \Spec A_n
\] 
and 
$o_{\xi_n} \in \Ext^2(\Omega^1_U,\mathcal{O}_U)$ the obstruction class defined in the proof of Proposition \ref{obs}. 
We show that $o_{\xi_n}=0$ in the following. 

Let $\omega_X$ be the dualizing sheaf on $X$. 
By taking the tensor product of the sequence (\ref{obsseq}) with the relative dualizing sheaf 
$\omega_{\mathcal{U}_n/\Spec A_n}$ of $f_{\mathcal{U}_n}$, we have an exact sequence
 \begin{multline}\label{wobsU}
 0 \rightarrow \omega_U \rightarrow f_{\mathcal{U}_n}^{\ast} (\Omega^1_{\Spec A_{n+1}/k}|_{\Spec A_n}) 
\otimes \omega_{\mathcal{U}_n/\Spec A_n} \\
 \rightarrow \Omega^1_{\mathcal{U}_n/k} \otimes \omega_{\mathcal{U}_n/\Spec A_n} \rightarrow 
\Omega^1_{\mathcal{U}_n/\Spec A_n} \otimes \omega_{\mathcal{U}_n/ \Spec A_n} \rightarrow 0.  
 \end{multline}

By taking $\iota_{\ast}$ of the above sequence, we get a sequence 
\begin{multline}\label{wobsX}
0 \rightarrow \omega_X \rightarrow \iota_{\ast} (f_{\mathcal{U}_n}^* \Omega^1_{\Spec A_{n+1}/k}|_{\Spec A_n} 
\otimes \omega_{\mathcal{U}_n/\Spec A_n}) \\
 \rightarrow \iota_* (\Omega^1_{\mathcal{U}_n/k} \otimes \omega_{\mathcal{U}_n/ \Spec A_n})
 \rightarrow \iota_{\ast} (\Omega^1_{\mathcal{U}_n/\Spec A_n} \otimes \omega_{\mathcal{U}_n/ \Spec A_n}) \rightarrow 0.  
\end{multline}
This sequence is exact by the following claim. 

\begin{claim}\label{omegaR1}  
\begin{enumerate}
\item[(i)] $R^1 \iota_{\ast} \omega_U =0$. 
\item[(ii)] $R^1 \iota_{\ast}( f_{\mathcal{U}_n}^{\ast}\Omega^1_{\Spec A_n/k} \otimes \omega_{\mathcal{U}_n/ \Spec A_n}) =0$.  
\end{enumerate}
\end{claim}

\begin{proof}[Proof of Claim]
\noindent(i) Let $p \in X \setminus U$ be a singular point and $U_p$ a small affine neighborhood at $p$. 
It is enough to show that $H^2_p (U_p, \omega_{U_p}) =0$. 
Let $\pi_p \colon V_p \rightarrow U_p$ be the index 1 cover of $U_p$. Then we have 
$(\pi_p)_* \mathcal{O}_{V_p} \simeq \oplus_{i=0}^{r-1} \mathcal{O}_{U_p}(iK_{U_p})$ where $r$ is the index of 
the singularity $p \in X$. Hence 
\[
H^2_q (V_p, \mathcal{O}_{V_p} ) \simeq \bigoplus_{i=0}^{r-1} H^2_p(U_p, \mathcal{O}_{U_p}(i K_{U_p})), 
\] 
where $q:= \pi^{-1}(p)$. 
  L.H.S. is zero by the same argument as in Claim \ref{R1O} since  ${\rm depth}_q \mathcal{O}_{V_p,q} = 3$. Hence we proved the first claim. 

\noindent (ii) Let $f_{(n,p)} \colon \mathcal{U}_{(n,p)} \rightarrow \Spec A_n$ be the deformation of $U_p$ induced from $f_n$. 
It is enough to show that 
\[
H^2_p (\mathcal{U}_{(n,p)},  f_{(n,p)}^* \Omega^1_{\Spec A_n / k} \otimes \omega_{\mathcal{U}_{(n,p)} / A_n}) =0 .
\]
Set $\omega^{[i]}_{\mathcal{U}_{(n,p)} / A_n } := \iota_* \omega^{\otimes i}_{{\mathcal{U}'}_{(n,p)} / A_n }$, where 
$\mathcal{U}'_{(n,p)} := \mathcal{U}_{(n,p)} \setminus \{p \}$.  
We can take an index 1 cover 
$\phi_{(n,p)} \colon \mathcal{V}_{(n,p)} \rightarrow \mathcal{U}_{(n,p)}$ which is determined by an isomorphism 
$\omega^{[r_p]}_{\mathcal{U}_{(n,p)} / A_n } \simeq \mathcal{O}_{\mathcal{U}_{(n,p)}}$, where $r_p$ is the Gorenstein index of $U_p$. 
Set $g_{(n,p)}:= f_{(n,p)} \circ \phi_{(n,p)}$.   
 Note that 
\[
(\phi_{(n,p)})_*(g_{(n,p)}^* \Omega^1_{\Spec A_n / k}) \simeq 
\bigoplus_{i=0}^{r-1} f_{(n,p)}^* \Omega^1_{\Spec A_n / k} \otimes \omega^{[i]}_{\mathcal{U}_{(n,p)} / A_n }.
\] 
We can see that $H^2_p (\mathcal{U}_{(n,p)},  f_{(n,p)}^* \Omega^1_{\Spec A_n / k} \otimes \omega_{\mathcal{U}_{(n,p)} / A_n})$ 
is a direct summand of 
\[
H^2_q(\mathcal{V}_{(n,p)}, g_{(n,p)}^* \Omega^1_{\Spec A_n /k}) \simeq H^2_q(\mathcal{V}_{(n-1,p)}, \mathcal{O}_{\mathcal{V}_{(n-1,p)}}) 
\]
and this is zero by Claim \ref{R1O}(ii).

\end{proof}

Since the sheaf $\iota_{\ast}(\Omega^1_{\mathcal{U}_n/ A_n} \otimes \omega_{\mathcal{U}_n/A_n})$ is flat over $A_n$ by Lemma \ref{iotaflatlem}, 
we have an isomorphism 
\[
 \Ext^2_{\mathcal{O}_{\mathcal{X}_n}} (\iota_{\ast}(\Omega^1_{\mathcal{U}_n/ A_n} \otimes \omega_{\mathcal{U}_n/A_n}), \omega_X ) \simeq
  \Ext^2_{\mathcal{O}_{X}} (\iota_{\ast}(\Omega^1_U \otimes \omega_U), \omega_X )
\]
by Lemma \ref{extisom}. 
By using this isomorphism, we define $o'_{\xi_n} \in \Ext^2_{\mathcal{O}_X}(\iota_{\ast}(\Omega^1_U \otimes \omega_U), \omega_X )$ to be the element corresponding 
to the sequence (\ref{wobsX}). 
Let $r_2 \colon \Ext^2_{\mathcal{O}_X}(\iota_{\ast}(\Omega^1_U \otimes \omega_U), \omega_X ) 
\rightarrow \Ext^2_{\mathcal{O}_U}( \Omega^1_U \otimes \omega_U, \omega_U)$ be the natural restriction map and 
 $T \colon \Ext^2_{\mathcal{O}_U}( \Omega^1_U \otimes \omega_U, \omega_U)
  \rightarrow \Ext^2_{\mathcal{O}_U}(\Omega^1_U,\mathcal{O}_U)$ be the map induced by tensoring $\omega_U^{-1} $. 
Then we have 
\[
T(r_2(o'_{\xi_n})) = o_{\xi_n}. 
\] 
Hence it is enough to show that $\Ext^2_{\mathcal{O}_X}(\iota_{\ast} (\Omega^1_U \otimes \omega_U),\omega_X)=0$. 
By the Serre duality, we have $\Ext^2_{\mathcal{O}_X}(\iota_{\ast} (\Omega^1_U \otimes \omega_U),\omega_X)^{\ast} \simeq 
H^1(X, \iota_{\ast} (\Omega^1_U \otimes \omega_U))$, where ${ }^*$ is the dual.

In the following, we show that \[
H^1(X, \iota_* (\Omega^1_U \otimes \omega_U)) =0.
\] 
Let $m$ be a positive integer such that $-m K_X$ is very ample and $| {-} m K_X |$ contains a smooth member $D_m$ which 
is disjoint with the singular points of $X$. Let $\pi_m \colon Y_m:= \Spec \oplus_{i=0}^{m-1} \mathcal{O}_X(i K_X) \rightarrow X$ 
be a cyclic cover determined by $D_m$. Note that $Y_m$ has only terminal Gorenstein singularities. 

There is the residue exact sequence 
\[
0 \rightarrow \Omega^1_U \rightarrow \Omega^1_U(\log D_m) \rightarrow \mathcal{O}_{D_m} \rightarrow 0
\]
By tensoring this sequence with $\omega_U$ and taking the push-forward of the sheaves by $\iota$, we obtain an exact sequence 
\[
0 \rightarrow \iota_* (\Omega^1_U \otimes \omega_U) \rightarrow \iota_* (\Omega^1_U (\log D_m) \otimes \omega_U) 
\rightarrow \iota_* (\omega_U |_{D_m}). 
\]
The last homomorphism is surjective and $\iota_*( \omega_U |_{D_m}) \simeq \omega_X|_{D_m}$ since $ \iota_* (\omega_U |_{D_m})$ is supported on $D_m \subset U$. 
Hence we obtain an exact sequence 
\begin{equation}
0 \rightarrow \iota_* (\Omega^1_U \otimes \omega_U) \rightarrow \iota_* (\Omega^1_U (\log D_m) \otimes \omega_U) 
\rightarrow \omega_X |_{D_m} \rightarrow 0
\end{equation}
 It induces an exact sequence 
\[
H^0 (X, \omega_X|_{D_m} ) \rightarrow H^1(X, \iota_* (\Omega^1_U \otimes \omega_U) ) \rightarrow 
H^1(X, \iota_* (\Omega^1_U(\log D_m) \otimes \omega_U) ). 
\]
We have $H^0 (X, \omega_X|_{D_m} ) =0 $ since $-K_X$ is ample. Therefore, it is enough to show that 
\[
H^1(X, \iota_* (\Omega^1_U(\log D_m) \otimes \omega_U) )=0.
\] 
Put $D' := \pi_m^{-1}(D_m)$ which satisfies that $D' \simeq D_m$ and $\pi_m^* D_m = m D'$. 
By using the isomorphism 
\[
(\pi_m)_{\ast}\left( \Omega^1_{Y_m}(\log D')(-D') \right) \simeq 
\bigoplus_{i=0}^{m-1} \iota_* \left( \Omega^1_U( \log D_m) \otimes \mathcal{O}_U\left( \left( i+1 \right) K_U \right)\right), 
\] 
we can see that $H^1(X, \iota_* (\Omega^1_U(\log D_m) \otimes \omega_U) )$ is a direct summand of 
\[
H^1 (Y_m, \Omega^1_{Y_m}(\log D')(-D')).
\]
We can show that 
\[
H^1 (Y_m, \Omega^1_{Y_m}(\log D')(-D'))=0
\]
as follows. There is an exact sequence 
\[
0 \rightarrow \Omega^1_{Y_m}(\log D')(-D') \rightarrow \Omega^1_{Y_m} \rightarrow \Omega^1_{D'} \rightarrow 0
\]  
and it induces an exact sequence 
\[
H^0(D', \Omega^1_{D'}) \rightarrow H^1(Y_m, \Omega^1_{Y_m}(\log D')(-D') ) \rightarrow 
H^1(Y_m,\Omega^1_{Y_m}) \stackrel{\beta}{\rightarrow} H^1(D', \Omega^1_{D'}).
\]
We can see that $H^1( D', \mathcal{O}_{D'}) =0$ since $D_m \simeq D'$ and we have an exact sequence 
\[
0 \rightarrow \mathcal{O}_X (-D_m) \rightarrow \mathcal{O}_X \rightarrow \mathcal{O}_{D_m} \rightarrow 0.  
\] 
This and the Hodge symmetry imply $H^0(D',\Omega^1_{D'})=0$. 
Hence it is enough to show that $\beta$ is injective. 
We use the following commutative diagram 
\[
\xymatrix{
H^1(Y_m, \Omega^1_{Y_m}) \ar[r]^{\beta} & H^1(D', \Omega^1_{D'}) \\
      H^1(Y_m, \mathcal{O}_{Y_m}^*) \ar[u]^{\phi} \ar[d]^{\beta_1} \otimes \mathbb{C} \ar[r]^{\gamma} 
& H^1(D', \mathcal{O}_{D'}^*) \otimes \mathbb{C} \ar[u]^{\psi} \ar[d]^{\beta_2} \\
     H^2(Y_m, \mathbb{C}) \ar[r]^{\delta} & H^2(D', \mathbb{C}).
   }
\]
We can see that $\delta$ is injective by Theorem \ref{lef} since $Y_m$ 
 has only l.c.i. singularities. Note that $\beta_1$ is an isomorphism since 
$H^i(Y_m,\mathcal{O}_{Y_m})= 0 $ for $i =1,2$. Hence $\delta \circ \beta_1 = \beta_2 \circ \gamma$ is injective.  
This implies that $\gamma$ is injective. We can show that $\phi$ is surjective by an argument 
which is similar to that in \cite[(2.2)]{Namtop}. Note that $\psi$ is injective since $D'$ is a smooth surface 
and $H^1(D', \mathcal{O}_{D'})=0$. Hence $\psi \circ \gamma = \beta \circ \phi$ is injective. 
Therefore $\beta$ is injective. 

Hence we proved $o_{\xi_n} =0$. It is enough for unobstructedness by Proposition \ref{enough} since $X$ is a projective variety 
and has a semi-universal deformation space.   
\end{proof}

\begin{rem}
For a Fano $3$-fold $X$ with canonical singularities, its Kuranishi space $\Def(X)$ is not smooth in general. 
For example, let $X$ be a cone over the del Pezzo surface of degree $6$. Then $X$ has 2 different smoothings $\mathbb{P}^1 \times \mathbb{P}^1 \times \mathbb{P}^1$ 
and $\mathbb{P}(\Theta_{\mathbb{P}^2})$ in Grothendieck's notation, where $\Theta_{\mathbb{P}^2}$ is the tangent sheaf.  
\end{rem}

\section{A $\mathbb{Q}$-smoothing of a $\mathbb{Q}$-Fano 3-fold: the ordinary case}\label{qsmqfano3sect}

\subsection{Stratification on the Kuranishi space of a singularity}\label{stratif}
First, we recall a stratification on the Kuranishi space of an isolated singularity introduced in the proof of   \cite[Theorem 2.4]{NamSt}. 

Let $V$ be a Stein space with an isolated hypersurface singularity $p \in V$. Then we have its semi-universal deformation space $\Def(V)$ and the semi-universal family 
$\mathcal{V} \rightarrow \Def(V)$.  
It has a stratification into Zariski locally closed and 
smooth subsets $S_k \subset \Def (V)$ for $k \ge 0$ with the following properties; 
\begin{itemize}
\item $\Def(V) =  \amalg_{k \ge 0} S_k$.  
\item $S_0$ is a non-empty Zariski open subset of $\Def(V)$ and $\mathcal{V}$ is smooth
over $S_0$. 
\item $S_k$ are of pure codimension in $\Def(V)$ for all $k > 0$ and $\codim_{ \Def(V)}
S_k < \codim_{ \Def(V)} S_{k+1}$. 
\item If $k> l$, then $\overline{S_k} \cap S_l = \emptyset$. 
\item $\mathcal{V}$ has a simultaneous resolution on each $S_k$, that is, there is a resolution of 
$\mathcal{V} \times_{\Def(V)} S_k$ which is smooth over $S_k$. 
\end{itemize}

\subsection{A useful homomorphism between cohomology groups}\label{useful}

Let us explain the homomorphism which we need for finding $\mathbb{Q}$-smoothings. 
Let $p \in U$ be a $3$-fold Stein neighborhood of a terminal singularity $p$ of index $r$, that is, 
$r$ is the minimal positive integer such that $rK_U$ is Cartier.  
Fix a positive integer $m$ such that $r | m$. Let  
\[
\pi_U \colon  V:= \Spec \bigoplus_{i=0}^{m-1} \mathcal{O}(i K_U) \rightarrow U
\] 
be the finite morphism defined by the isomorphism $\mathcal{O}_U(r K_U) \simeq \mathcal{O}_U$. Note that $V$ is a disjoint union of 
several copies of the index 1 cover of $U$. 
Let $G:= \mathbb{Z}/ m \mathbb{Z}$ be the Galois group of $\pi_U$.  
Set $Q := \pi_U^{-1}(p)$. 

We consider the case $m=r$ to explain the ordinariness of a terminal singularity. 
In this case, $V$ is called the {\it index one cover} of $U$. The germ $(V, Q)$ is a germ of a terminal Gorenstein singularity
and it is known that $(V,Q)$ is a cDV singularity and that $(V,Q)$ is a hypersurface in the germ $(\mathbb{C}^4,0)$.  
We can embed $(V,Q)$ in $(\mathbb{C}^4, 0)$ in such a way that the $\mathbb{Z}_r$-action on $(V,Q)$ extends to a $\mathbb{Z}_r$-action 
on $(\mathbb{C}^4,0)$. Moreover, we may assume that $(V,Q)$ is a hypersurface defined by a $\mathbb{Z}_r$-semi-invariant function $f_V$.  
 Let $\zeta_U \in \mathbb{C}$ be the eigenvalue of the action on $f_V$, that is, $\zeta_U$ satisfies that 
 $g \cdot f_V = \zeta_U f_V$, where $g \in G$ is the generator.   
We have the following fact by the classification of $3$-fold 
terminal singularities by Reid and Mori. 

\begin{fact}
Let $(U,p)$ be a germ of a $3$-fold terminal singularity. Then $\zeta_U$ is $1$ or $-1$. 
\end{fact} 
By this fact, we introduce the following notions on terminal singularities. 

\begin{defn}
Let $(U,p)$ be a germ of $3$-fold terminal singularity. 
We say that $(U,p)$ is {\it ordinary} (resp.\ {\it exceptional}) if $\zeta_U =1$ (resp.\ $\zeta_U =-1$).  
\end{defn}

Now we go back to general $m$ which is some multiple of $r$. 
Let $\nu_V \colon \tilde{V} \rightarrow V$ be a $G$-equivariant good resolution, $F_V:= \nu_V^{-1}(Q)= \Exc (\nu_V) $ its exceptional locus 
which has normal crossing support  
and $\tilde{U} := \tilde{V}/ G$ the quotient. So we have a diagram 

\begin{equation}\label{localdiag1}
\xymatrix{
      \tilde{V} \ar[r]^{\tilde{\pi}_U} \ar[d]^{\nu_V} & \tilde{U} \ar[d]^{\mu_U} \\
      V \ar[r]^{\pi_U} & U . 
     }
 \end{equation} 
Let  
$\mathcal{F}_U^{(0)}$ be the $\mathbb{Z}_m$-invariant part of 
$(\tilde{\pi}_U)_* (\Omega^2_{\tilde{V}}(\log F_V)(-F_V-\nu_V^* K_V))$. 
Set $V':= V \setminus Q$. 
We have the coboundary map of the local cohomology group 
\[
\tau_V \colon H^1(V', \Omega_{V'}^2 \otimes \omega_{V'}^{-1}) \rightarrow 
H^2_{F_V}( \tilde{V}, \Omega_{\tilde{V}}^2( \log F_V)(-F_V -\nu_V^* K_V) ). 
\]
This is same as the homomorphism used by Namikawa--Steenbrink \cite{NamSt} and Minagawa \cite{mina}. 

\begin{lem}\label{nonzerocobdry}(\cite[Theorem 1.1]{NamSt}, \cite[Lemma 4.1]{mina})
Let $V$ be a Stein space as above. 
Assume that $V$ is not rigid. Then $\tau_V \neq 0$. 
\end{lem}

We see that the cohomology groups appearing in $\tau_V$ are $\mathcal{O}_{V, Q}$-modules. 
Moreover, $\tau_V$ is an $\mathcal{O}_{V,Q}$-module homomorphism. Note that $T^1_{(V,Q)} \simeq  H^1(V', \Omega_{V'}^2 \otimes \omega_{V'}^{-1})$ 
is generated by one element $\eta_V$ as an $\mathcal{O}_{V,Q}$-module. Actually $\eta_V \in T^1_{(V,Q)}$ corresponds to a deformation 
$(f_V + t =0) \subset (\mathbb{C}^4,0) \times \Delta^1$, where $t$ is the coordinate on $\Delta^1$.    
Hence we see that $\tau_V(\eta_V) \neq 0$. 

The $G$-invariant part of $\tau_V$ is 
\[
\phi_U \colon H^1( U', \Omega^2_{U'} \otimes \omega_{U'}^{-1}) \rightarrow H^2_{E_U}( \tilde{U}, \mathcal{F}^{(0)}_{U}),     
\]
where $U':= U \setminus \{p \}$ is the punctured neighborhood and $E_U \subset \tilde{U}$ is the exceptional locus of $\mu_U$. 

If $(U,p)$ is ordinary, we see that $\eta_V$ is contained in $H^1( U', \Omega^2_{U'} \otimes \omega_{U'}^{-1}) \subset H^1(V', \Omega_{V'}^2 \otimes \omega_{V'}^{-1})$ 
since $\eta_V$ induces a deformation $(f_V + t =0)/ \mathbb{Z}_r \subset \mathbb{C}^4/\mathbb{Z}_r \times \Delta^1$ of the germ $(U,p)$. Hence we obtain the following. 

\begin{lem}\label{ordinarynonzero}
Let $(U,p)$ be a germ of an ordinary terminal singularity. Then $\phi_U \neq 0$. 
\end{lem} 

\subsection{Proof of Theorem \ref{qsmqfanointro}}\label{proofordqsm}

We can find  good first order deformations as follows. 

\begin{thm}\label{qsmqfano}
Let $X$ be a $\mathbb{Q}$-Fano $3$-fold. 

Then $X$ has a deformation $f \colon \mathcal{X} \rightarrow \Delta^1$ over an unit disc such that 
the singularities on $\mathcal{X}_t$ for $t \neq 0$ satisfy the following condition; 

Let  $p_t \in \mathcal{X}_t$ be a singular point and $U_{p_t}$ its Stein neighborhood. Then $\phi_{U_{p_t}} =0$,  
where $\phi_{U_{p_t}}$ is the homomorphism defined in Section $\ref{useful}$.   
\end{thm}

Lemma \ref{ordinarynonzero} and Theorem \ref{qsmqfano} imply the following. 

\begin{cor}\label{qsmordqfano}
Let $X$ be a $\mathbb{Q}$-Fano $3$-fold with only ordinary terminal singularities. 
Then $X$ has a $\mathbb{Q}$-smoothing.
\end{cor}

\begin{proof}[Proof of Corollary \ref{qsmordqfano}]
By Lemma \ref{ordinarynonzero}, we can continue the process in the proof of Theorem \ref{qsmqfano} until we get 
a $\mathbb{Q}$-smoothing since deformations of ordinary terminal singularities are ordinary. 
\end{proof}

\begin{rem}
We first explain the strategy of the proof of Theorem \ref{qsmqfano}. 
Let $p_i \in U_i$ be a Stein neighborhood of a singularity on $X$. 
In order to find a good deformation direction, we study the restriction homomorphism $ p_{U_i} \colon T^1_X \rightarrow T^1_{U_i}$. 
The problem is that this is not always surjective. Actually there is an example of a $\mathbb{Q}$-Fano $3$-fold $X$ such that 
$H^2(X, \Theta_X) \neq 0$ (\cite[Example 5]{NamFano}). 
So we use the commutative diagram as in (\ref{minadiag}). 
The diagram is similar to that in the proof of \cite[Theorem 4.2]{mina}. 
Minagawa used a cyclic cover of $X$ branched only on singular points. 
We use a cyclic cover of $X$ branched along a divisor, but the framework of the proof is almost same.   
\end{rem}

\begin{proof}[Proof of Theorem \ref{qsmqfano}]
Let $p_1,\ldots, p_l \in X$  be the non-rigid singular points of $X$ such that
 $p_1,\ldots, p_{l'}$ for some $l' \le l$ are the points which satisfy  
\[
\phi_{U_i} \neq 0 
\]  
 for $i=1, \ldots, l'$, where $U_i$ is a  small Stein neighborhood of $p_i$.  

First we prepare notations to introduce the diagram (\ref{minadiag}).  
Let $m$ be a sufficiently large integer such that  $-mK_X $ is very ample and $|{-}m K_X|$ contains a smooth member $D_m$ 
such that $D_m \cap \Sing X = \emptyset$. 
 Let 
\[
\pi \colon Y:= \Spec \bigoplus_{i=0}^{m-1} \mathcal{O}_X(iK_X)  \rightarrow X
\]
 be the cyclic cover determined by $D_m$. 
There exists a good $\mathbb{Z}_m$-equivariant resolution (\cite{AW}) $\nu \colon  \tilde{Y} \rightarrow Y$ which induces an isomorphism 
$\nu^{-1}(Y \setminus \pi^{-1} \{p_1,\ldots, p_l \}) \rightarrow Y \setminus \pi^{-1} \{p_1,\ldots, p_l \}$ 
and a birational morphism 
$\mu \colon \tilde{X} := \tilde{Y}/ \mathbb{Z}_m \rightarrow X$.  
These induce the following cartesian diagram; 
\begin{equation}\label{globaldiag}
\xymatrix{
      \tilde{Y} \ar[r]^{\tilde{\pi}} \ar[d]^{\nu} & \tilde{X} \ar[d]^{\mu} \\
     Y \ar[r]^{\pi}& X .
     }
 \end{equation}

Let $\pi_i \colon V_{i} := \pi^{-1}(U_i) \rightarrow U_i$ and 
$\nu_i \colon \tilde{V}_{i}:= \nu^{-1}(V_i) \rightarrow V_{i}$ be morphisms induced by the morphisms in the above diagram. 
Put $\tilde{U}_i := \tilde{V}_{i}/ \mathbb{Z}_m$.  Then we get the following cartesian diagram;  
\begin{equation}\label{localdiag}
\xymatrix{
      \tilde{V}_{i} \ar[r]^{\tilde{\pi}_{i}} \ar[d]^{\nu_{i}} & \tilde{U}_i \ar[d]^{\mu_i} \\
      V_{i} \ar[r]^{\pi_{i}} & U_i .
     }
 \end{equation}
 
 Put $F:= \Exc(\nu), E:= \Exc (\mu),  D' := \pi^{-1}(D_m) $ and $ L' := \mathcal{O}_{Y}(D') = \mathcal{O}_Y(\pi^*(-K_X))$. Note that 
$F$ has normal crossing support since $\nu$ is good. 
Also put $F_{i} := \Exc (\nu_{i})$ and 
$E_i := \Exc (\mu_i)$.  
 Let $\mathcal{F}^{(0)}$ be the $\mathbb{Z}_m$-invariant part of 
 $
 \tilde{\pi}_* (\Omega^2_{\tilde{Y}}(\log F)(-F) \otimes \nu^* L')$.  
 Let $U$ be the smooth part of $X$. 
Note that $\mathcal{F}^{(0)}|_U \simeq \Omega^2_U \otimes \omega_U^{-1}$.  
 Set $\mathcal{F}_i^{(0)} := \mathcal{F}^{(0)}|_{\tilde{U}_i}$ 
and $U'_i := U_i \setminus \{p_i \}$.
Note that $\mathcal{F}_i^{(0)}|_{U'_i} \simeq \Omega^2_{U'_i} \otimes \omega_{U'_i}^{-1}$.

 We have the following commutative diagram;  
\begin{equation}\label{minadiag}
\xymatrix{
      H^1(U, \Omega^2_U \otimes \omega_U^{-1} ) \ar[r]^{\oplus \psi_i} \ar[d]^{\oplus p_{U_i}} & 
\oplus_{i=1}^{l'} H^2_{E_i}(\tilde{X}, \mathcal{F}^{(0)} )
\ar[d]^{\simeq} \ar[r] &  H^2 (\tilde{X}, \mathcal{F}^{(0)} ) \\
     \oplus_{i=1}^{l'} H^1(U'_i, \Omega^2_{U'_i}\otimes \omega_{U'_i}^{-1}) \ar[r]^{\oplus \phi_{i}}& 
\oplus_{i=1}^{l'}  H^2_{E_i}(\tilde{U}_i, \mathcal{F}_{i}^{(0)}). 
     }
 \end{equation}
We identify $H^2_{E_i}(\tilde{X}, \mathcal{F}^{(0)} )$ and $H^2_{E_i}(\tilde{U}_i, \mathcal{F}_{i}^{(0)})$ by the natural homomorphism induced by restriction. 
Note that $
\mathcal{F}_i^{(0)} \simeq \mathcal{F}_{U_i}^{(0)}  
$, where $\mathcal{F}_{U_i}^{(0)}$ is the sheaf defined in Section \ref{useful}.
Hence $\phi_i$ is $\phi_{U_i}$ in Section \ref{useful}.

Next we see that $p_{U_i}$ in the diagram (\ref{minadiag}) is the restriction homomorphism of $T^1$ as follows. 
Let $T^1_X, T^1_{V_i}, T^1_{U_i}$ be the tangent spaces of the functors $\Def_X, \Def_{V_i}, \Def_{U_i}$ respectively. 
By \cite[\S 1 Theorem 2]{Schl}  or the proof of Proposition \ref{obs} in this paper, we can see that the first order deformations of 
$V_i, U_i$ are bijective to those of the smooth part $V'_i, U'_i$. Similarly we can see the same correspondence for $X$. So we have    
\[
T^1_X \simeq H^1(U, \Theta_U) \simeq H^1(U, \Omega_U^2 \otimes \omega_U^{-1}), 
\]
\[ 
T^1_{V_i} \simeq H^1(V'_i, \Theta_{V'_i}) \simeq H^1(V'_i, \Omega^2_{V'_i} \otimes \omega^{-1}_{V'_i}),    
\]
\[ 
T^1_{U_i} \simeq H^1(U'_i, \Theta_{U'_i}) \simeq H^1(U'_i, \Omega^2_{U'_i} \otimes \omega^{-1}_{U'_i}),    
\]
where $\Theta_U, \Theta_{V'_i}, \Theta_{U'_i}$ are the tangent sheaves of $U, V'_i, U'_i$ respectively. 
Hence $p_{U_i}$ is regarded as the restriction homomorphism $T^1_X \rightarrow T^1_{U_i}$.  

We want to lift $\eta_i \in H^1(U'_i, \Omega^2_{U'_i} \otimes \omega^{-1}_{U'_i}) \simeq T^1_{U_i}$ which induces a non-trivial deformation of $U_i$ 
to an element of $H^1(U, \Omega_U^2 \otimes \omega_U^{-1}) \simeq T^1_X$. 
In order to do that, we consider $\phi_i(\eta_i) \in H^2_{E_i}(\tilde{U}_i, \mathcal{F}_{i}^{(0)})$ and lift it by using the diagram (\ref{minadiag}).

Since $\tilde{\pi}$ is finite, $H^2(\tilde{X}, \mathcal{F}^{(0)})$ is a direct summand of 
\[
H^2(\tilde{Y}, \Omega^2_{\tilde{Y}}(\log F)(-F) \otimes \nu^* L')
\]  
and this is zero by the vanishing theorem by Guillen-Navarro Aznar-Puerta-Steenbrink (\cite{PS} Theorem 7.30 (a)). Hence $\oplus \psi_i$ is surjective.

By the assumption that $\phi_i \neq 0$ for $i=1, \ldots, l'$, there exists 
$\eta_i \in H^1( U'_i, \Omega^2_{U'_i} \otimes \omega_{U'_i}^{-1}) \setminus \Ker \phi_{i} $. 
By the surjectivity of $\oplus \psi_i$, there exists $\eta \in H^1(U, \Omega^2_{U}\otimes \omega_U^{-1} )$ such that 
$\psi_i(\eta) = \phi_i ( \eta_i)$. 
Then we have  
\begin{equation}\label{etaker}
p_{U_i} ( \eta ) \notin \Ker (\phi_{i}).
\end{equation} 
We want to see that $p_{U_i}(\eta)$ induces a non-trivial deformation of a singularity $p_i \in U_i$.  
For that purpose, we study the deformation of $V_i$ induced by $p_{U_i}(\eta)$ and see that it does not come from 
a deformation of the resolution of $V_i$.

Since $V_i$ has only rational singularities, the birational morphism $\nu_{i} \colon \tilde{V}_{i} \rightarrow V_{i}$ induces a morphism 
of the functors 
$\Def_{\tilde{V}_{i}} \rightarrow \Def _{V_{i}}$ (\cite[Theorem 1.4 (c)]{wahl}) and the homomorphism 
$H^1(\tilde{V}_{i}, \Theta_{\tilde{V}_{i}}) \rightarrow H^1(V'_{i}, \Theta_{V'_{i}})$ on their tangent spaces. 
This homomorphism can be rewritten as 
\[
(\nu_{i})_* \colon H^1(\tilde{V}_{i}, \Omega^2_{\tilde{V}_{i}}\otimes \omega_{\tilde{V}_{i}}^{-1}) \rightarrow 
H^1(V'_{i}, \Omega^2_{V'_{i}} \otimes \omega^{-1}_{V'_i})  
\] 
and this is a homomorphism induced by an open immersion $V'_i \hookrightarrow \tilde{V}_i$.   
Note that infinitesimal deformations of $U_i$ come from $\mathbb{Z}_m$-equivariant deformations of $V_{i}$ and 
 $H^1( U'_i, \Theta_{U'_i}) \simeq H^1( V'_{i}, \Theta_{V'_{i}})^{\mathbb{Z}_m}$.     

Note that $\phi_i$ is the $\mathbb{Z}_m$-invariant part of the homomorphism 
\[
\tau_i \colon H^1(V'_i, \Omega^2_{V'_i} \otimes \omega^{-1}_{V'_i}) 
\rightarrow H^2_{F_i}(\tilde{V}_i, \Omega_{\tilde{V}_i}^2(\log F_i)(-F_i - \nu_i^*K_{V_i}) ). 
\]

\begin{claim}\label{nones}
$\Image (\nu_{i})_* \subset \Ker \tau_{i}$.
\end{claim}

\begin{proof}[Proof of Claim] 
We can write 
\[
K_{\tilde{V}_{i}} = \nu_{i}^* K_{V_{i}} + \sum_{j=1}^{m_i} a_{i,j} F_{i,j},
\]
where $F_i = \bigcup_{j=1}^{m_i} F_{i,j}$ is the irreducible decomposition and $a_{i,j} \ge 1$ are some integers 
for $j=1, \ldots, m_i$ since $V_i$ is terminal Gorenstein.  
We can define a homomorphism 
\[
\alpha_i \colon H^1( \tilde{V}_i, \Omega_{\tilde{V}_i}^2 \otimes \omega^{-1}_{\tilde{V}_i}) 
\rightarrow H^1(\tilde{V}_i, \Omega^2_{\tilde{V}_i}(\log F_i)(-F_i-\nu_i^*K_{V_i}) )
\] 
as a composite of the following homomorphisms; 
\begin{multline}
\alpha_i \colon H^1( \tilde{V}_i, \Omega_{\tilde{V}_i}^2 \otimes \omega^{-1}_{\tilde{V}_i}) = 
H^1( \tilde{V}_i, \Omega_{\tilde{V}_i}^2 ( -\sum_{j=1}^{m_i} a_{i,j} F_{i,j} - \nu_i^* K_{V_i})  ) \\ 
\rightarrow H^1(\tilde{V}_i, \Omega^2_{\tilde{V}_i}(\log F_i)(-\sum_{j=1}^{m_i} a_{i,j} F_{i,j} - \nu_i^*K_{V_i}))
\rightarrow H^1(\tilde{V}_i, \Omega^2_{\tilde{V}_i}(\log F_i)(-F_i - \nu_i^*K_{V_i})) 
\end{multline}
since $a_{i,j} \ge 1$. 

Note that $\Ker \tau_i = \Image \rho_i$, where we put 
\[
\rho_i \colon H^1(\tilde{V}_i, \Omega_{\tilde{V}_i}^2(\log F_i)(-F_i - \nu_i^*K_{V_i}) ) 
\rightarrow H^1(V'_i, \Omega^2_{V'_i} \otimes \omega^{-1}_{V'_i}).
\] 
We can see that $(\nu_i)_*$ factors as 
\[
(\nu_i)_* \colon H^1(\tilde{V}_i, \Omega^2_{\tilde{V}_i} \otimes \omega^{-1}_{\tilde{V}_i}) 
 \stackrel{\alpha_i}{\rightarrow} H^1(\tilde{V}_i, \Omega^2_{\tilde{V}_i}(\log F_i)(-F_i - \nu_i^*K_{V_i}) ) 
\stackrel{\rho_i}{\rightarrow} H^1(V'_i, \Omega^2_{V'_i}\otimes \omega^{-1}_{V'_i} ). 
\]
Hence $\Ker \tau_i = \Image \rho_i \supset \Image (\nu_i)_*$.   
\end{proof}

By Claim \ref{nones} and the relation (\ref{etaker}), we get $p_{U_i} (\eta) \not\in \Image (\nu_{i})_*$. 
This means that a deformation of $V_i$ induced by $p_{U_i} (\eta)$ does not come from that of the resolution $\tilde{V}_i$. 
In the following, we check that the deformation of $V_i$ goes out from the minimal stratum of the stratification on the Kuranishi space $\Def(V_i)$ 
 introduced in Section \ref{stratif}. 

Let $r_i$ be the Gorenstein index of the singular point $p_i$ and $\pi_{i}^{-1} (p_i) =:\{ q_{i 1}, \ldots, q_{i k(i)}\}$ , where $k(i):= \frac{m}{r_i}$. 
Let 
\[
V_i := \amalg_{j=1}^{k(i)} V_{i,j}
\]
be the decomposition into the connected components of $V_i$. 
Fix a stratification on each $\Def(V_{i,j})$ for $j=1,\ldots, k(i)$ as in Section \ref{stratif}. 
We see that  $p_{U_i}(\eta) \in T^1_{U_i} \subset T^1_{V_{i,1}}$ induces a deformation 
$g_{i,1} \colon \mathcal{V}_{i,1} \rightarrow \Delta^1$. 
By the property of the Kuranishi space, there exists a holomorphic map 
$\varphi_{i,1} \colon \Delta^1 \rightarrow \Def(V_{i,1})$ which induces the above deformation of $V_{i,1}$.  
Let $S_{i,k}$ be the minimal stratum of $\Def(V_{i,1})$. Then the image of $\varphi_{i,1}$ is not contained in $S_{i,k}$. and, 
for general $t \in \Delta^1$, 
 we have $\varphi_{i,1}(t) \in S_{i,k'}$ for some $k' < k$. 
 Let $g \colon \mathcal{X} \rightarrow \Delta^1$ be a small deformation of $X$ over a disc induced by $\eta \in H^1(U, \Theta_U)$. 
 Then $g$ induces a deformation of $V_{i,1}$ 
 We can continue this process as long as $\phi_i \neq 0$ and reach a deformation of $X$ whose general fiber has the required condition in the statement of Theorem \ref{qsmqfano}.  
\end{proof}

\begin{rem}
The author does not know $\phi_U$ is zero or not when $U$ is a Stein neighborhood of an exceptional terminal singularity. 
If we can prove $\phi_U \neq 0$ in that case, it implies Conjecture \ref{qsmqfanointro} by the above proof of Theorem \ref{qsmqfano}. 
\end{rem}

\begin{rem}
There is an example of a weak Fano $3$-fold which does not have a smoothing. 
It is written in \cite[Example 3.7]{minaosaka}. 
\end{rem}

\section{A $\mathbb{Q}$-smoothing of a $\mathbb{Q}$-Fano $3$-fold with a Du Val elephant}

In this section, we study the simultaneous $\mathbb{Q}$-smoothing problem as described in Conjecture \ref{simulqsmconj}. 

\subsection{Deformations of a $\mathbb{Q}$-Fano $3$-fold and its pluri-anticanonical element}

In this section, we prove unobstructedness of deformations of a $\mathbb{Q}$-Fano 3-fold with its pluri-anticanonical element. 
For that purpose, we first prepare a deformation functor of a pair of a Stein neighborhood of a terminal singularity and 
its $\mathbb{Q}$-Cartier divisor.  

Let $U$ be a Stein neighborhood of a $3$-fold terminal singularity of Gorenstein index $r$ and 
$D$ a $\mathbb{Q}$-Cartier divisor on $U$. 
We have the index one cover $\pi_U \colon V:= \Spec \oplus_{j=0}^{r-1} \mathcal{O}_U(jK_U) \rightarrow U$ 
determined by an isomorphism $\mathcal{O}_U(rK_U) \simeq \mathcal{O}_U$. 
Let $G:= \Gal (V/U)\simeq \mathbb{Z}_r$ be the Galois group of $\pi_U$. 
 This induces a $G$-action 
on the pair $(V, \Delta)$, where $\Delta:= \pi_U^{-1}(D)$.  
We can define  functors of $G$-equivariant deformations of $(V, \Delta)$ as follows. 

\begin{defn}
Let $\Def^G_{(V,\Delta)} \colon (Art_{\mathbb{C}}) \rightarrow (Sets)$ be a functor such that, for $A \in (Art_{\mathbb{C}})$, a set   
$\Def_{(V,\Delta)}^G(A) \subset \Def_{(V, \Delta)}(A)$ is the set of deformations $(\mathcal{V}, \mathbf{\Delta})$ of $(V, \Delta)$ over $A$ 
with a $G$-action which is compatible with the $G$-action on $(V, \Delta)$.   

We can also define the functor $\Def^G_{V} \colon (Art_{\mathbb{C}}) \rightarrow (Sets)$ of $G$-equivariant deformations of $V$ similarly. 
\end{defn}

\begin{prop}\label{Gequivdeffuncprop}
We have isomorphisms of functors 
\begin{equation}\label{Gequivfunctisom}
 \Def^G_{(V,\Delta)} \simeq \Def_{(U,D)}, \ \ \ \  \Def^G_{V} \simeq \Def_U. 
\end{equation} 
Moreover, these functors are unobstructed and the forgetful homomorphism 
$\Def_{(U,D)} \rightarrow \Def_U$ is a smooth morphism of functors.  
\end{prop}

\begin{rem}
	The latter isomorphism $\Def^G_{V} \simeq \Def_U$ is given  in \cite[Proposition 3.1]{Namtop}.
	\end{rem}

\begin{proof}
For a $G$-equivariant deformation of $(V, \Delta)$, we can construct a deformation of $(U,D)$ by taking its quotient by $G$. 
Conversely, given a deformation $(\mathcal{U}, \mathcal{D})$ of $(U,D)$. 
Let $\iota \colon U':= U \setminus \{ p \} 
\hookrightarrow U$ be an open immersion and $\mathcal{U}' \rightarrow \Spec A$ a deformation of $U'$ induced by $\mathcal{U}$.
Let $\omega_{\mathcal{U}/A}^{[i]}:= \iota_* \omega^{\otimes i}_{\mathcal{U}'/A}$. 
This is flat over $A_n$ by \cite[Theorem 12]{kollarflat}.    
 Thus we can construct a $G$-equivariant deformation of $(V,\Delta)$ by 
\[
\pi_{\mathcal{U}} \colon \mathcal{V}:= \Spec_{\mathcal{U}} \oplus_{i=0}^{r-1} \omega_{\mathcal{U}/A}^{[i]} \rightarrow \mathcal{U} 
\] 
and $\mathbf{\Delta}:= \pi_{\mathcal{U}}^*(\mathcal{D}) = \mathcal{D} \times_{\mathcal{U}} \mathcal{V}$, where $\pi_{\mathcal{U}}$ is defined by an isomorphism 
$\varphi_{s_{\mathcal{U}}} \colon \omega^{[r]}_{\mathcal{U}/A} \simeq \mathcal{O}_{\mathcal{U}}$ 
for some nowhere vanishing section $s_{\mathcal{U}} \in H^0(\mathcal{U}, \omega_{\mathcal{U}/A}^{[r]})$. 
Note that $\pi_{\mathcal{U}}$ is independent of the choice of a section $s_{\mathcal{U}}$. 
We can check that these constructions are converse to each other. Thus we obtain the required isomorphisms of 
functors. 

Since $V$ has only l.c.i.\ singularities and $\Delta$ is its Cartier divisor, we see the latter statements. 
Thus we finish the proof of Proposition \ref{Gequivdeffuncprop}. 
\end{proof}

By these local descriptions, we can show the following unobstructedness of a pair of 
a $\mathbb{Q}$-Fano $3$-fold and its pluri-anticanonical element. 

\begin{thm}\label{pairunobs}
Let $X$ be a $\mathbb{Q}$-Fano $3$-fold and $m$ a positive integer. Assume that $|{-} m K_X|$ contains an element $D$. 
Let $\Def_{(X,D)}$ and $\Def_X$ be the deformation functors of the pair $(X,D)$ and $X$ respectively. 

Then the forgetful map $\Def_{(X,D)} \rightarrow \Def_X$ is a smooth morphism of functors. 
In particular, the deformations of the pair $(X,D)$ are unobstructed. 
\end{thm}

\begin{proof}
Set $k:=\mathbb{C}$.  Let $A$ be an Artin local $k$-algebra, $e=( 0 \rightarrow k \rightarrow \tilde{A} \rightarrow A \rightarrow 0)$ a 
small extension  
and $\zeta:=( f \colon (\mathcal{X},\mathcal{D}) \rightarrow {\rm Spec} A)$ a flat deformation of the pair $(X,D)$.  
Assume that we have a lifting $\tilde{\mathcal{X}} \rightarrow \Spec \tilde{A}$ of $f \colon \mathcal{X} \rightarrow \Spec A$. 
It is enough to show that there exists a lifting $\tilde{\mathcal{D}} \subset \tilde{\mathcal{X}}$ of $\mathcal{D} \subset \mathcal{X}$. 
Let $\mathcal{I}_D \subset \mathcal{O}_X$ be the ideal sheaf of $D \subset X$ 
and $\mathcal{N}_{D/X}:= (\mathcal{I}_D/\mathcal{I}_{D}^2)^{\vee}$ be the normal sheaf of $D \subset X$. 
By Proposition \ref{Gequivdeffuncprop}, there exists a local lifting of $\mathcal{D}$ over $\tilde{A}$.  
Thus the condition in \cite[Theorem 6.2(b)]{HartDef} is satisfied and 
we see that an obstruction to the existence of a global lifting $\tilde{D}$ lies in $H^1(D, \mathcal{N}_{D/X})$. 
Hence it is enough to show that 
\[
H^1(D, \mathcal{N}_{D/X})=0.
\]  

Let $U$ be the smooth locus of $X$, $D_U:= D \cap U$ and $\mathcal{N}_{D_U/U}$ the normal sheaf of $D_U \subset U$. 
There is an exact sequence 
\[
0 \rightarrow \mathcal{O}_U \rightarrow \mathcal{O}_U(D_U) \rightarrow \mathcal{N}_{D_U/U} \rightarrow 0.   
\]
By taking the push forward by the open immersion $\iota$, we obtain an exact sequence 
\[
0 \rightarrow \mathcal{O}_X \rightarrow \mathcal{O}_X (D) \rightarrow \mathcal{N}_{D/X} \rightarrow 0 
\] 
since the sheaves $\mathcal{O}_X(D)$ and $\mathcal{N}_{D/X}$ are reflexive and we have 
$R^1 \iota_* \mathcal{O}_U =0$ by $\depth_p X =3$ for all $p \in X \setminus U$. 
This exact sequence induces an exact sequence 
\[
H^1(X, \mathcal{O}_X(D)) \rightarrow H^1(D, \mathcal{N}_{D/X}) \rightarrow H^2(X,\mathcal{O}_X). 
\]
The L.H.S and R.H.S. are zero by the Kodaira vanishing theorem. Hence we have $H^1(D, \mathcal{N}_{D/X})=0$. 
\end{proof}

\subsection{Existence of an essential resolution of a pair}\label{essresol}
We need a suitable resolution of a $3$-fold cDV singularity and its divisor with only Du Val singularities as follows. 

\begin{prop}\label{essresolprop}
Let $Y$ be a $3$-fold with only hypersurface singularities and $D$ a Cartier divisor on $Y$ with only Du Val singularities. 
Assume that a finite group $G$ acts on $Y$ and the action preserves $D$. 

Then there exists a $G$-equivariant resolution of singularities $f \colon \tilde{Y} \rightarrow Y$ of $Y$ with the following properties; 
\begin{enumerate} 
\item[(i)] The strict transform $\tilde{D} \subset \tilde{Y}$ of $D$ is smooth, 
\item[(ii)] We have $K_{\tilde{D}} = f_D^* K_D$, where $f_D \colon \tilde{D} \rightarrow D$ is the morphism induced by $f$. 
\end{enumerate}
\end{prop}

\begin{proof} 
Let \[
f_D \colon D_l \stackrel{f_{D_{l-1}}}{\rightarrow} D_{l-1} \rightarrow \cdots \rightarrow  D_1 \stackrel{f_{D_0}}{\rightarrow} D_0=D
\] be the minimal resolution of $D$, where $f_{D_i} \colon D_{i+1} \rightarrow D_i$ is a blow-up at a Du Val point $p_i \in D_i$ 
for $i=0, \ldots, l-1$. 
Let \[
f_Y \colon Y_l \stackrel{f_{Y_{l-1}}}{\rightarrow} Y_{l-1} \rightarrow \cdots \rightarrow  Y_1 \stackrel{f_{Y_0}}{\rightarrow} Y_0=Y  
\] 
be a composition of the blow-ups at the same smooth points as $f_D$. 
The surface $D_i$ can be regarded as a divisor on $Y_i$.  

\begin{claim} 
The divisor $D_i$ is Cartier on $Y_i$ for $i=1, \ldots, l$. 	
\end{claim}

\begin{proof}[Proof of Claim]
First, note that, if $Y$ is smooth at a Du Val singularity of $D$, we see the claim over that point. 
Thus we assume that $Y$ is singular. 

Since we can check the statements locally around a Du Val singularity of $D$, 
we may assume that $Y$ is embedded in  $Z:=\mathbb{C}^4$ as a Cartier divisor and 
there exists a divisor $\Delta \subset Z$ such that $D = \Delta \cap Y$. 
We may also assume that the defining equation of $Y$ is of the form 
\begin{equation}\label{cDVequationform}
g(x,y,z) + u h(x,y,z,u),  
\end{equation}
and $\Delta= (u=0) \subset Y$, where $g \in \mathbb{C}[x,y,z]$ is a defining equation of the Du Val singularity of $D$ 
and $h \in  \mathbb{C}[x,y,z,u]$ is a polynomial which vanishes on $p_0 \in D_0=D$ since $Y$ is singular at $p_0$. 

Let
\[
f_{\Delta} \colon \Delta_l \stackrel{f_{\Delta_{l-1}}}{\rightarrow} \Delta_{l-1} \rightarrow \cdots \rightarrow  \Delta_1 \stackrel{f_{\Delta_0}}{\rightarrow} \Delta_0=\Delta, 
\]
\[
f_Z \colon Z_l \stackrel{f_{Z_{l-1}}}{\rightarrow} Z_{l-1} \rightarrow \cdots \rightarrow  Z_1 \stackrel{f_{Z_0}}{\rightarrow} Z_0=Z 
\]
be compositions of the blow-ups at the same smooth points as $f_D$. 
Note that $Y_i, \Delta_i \subset Z_i$ and $D_i \subset Y_i \cap \Delta_i$. 
Let $E_i:=f^{-1}_{Z_{i-1}}(p_i) \subset Z_i$ be the exceptional divisor. 

We can check that $D_1 \subset Y_1$ is a Cartier divisor as follows; 
It is enough to check that $D_1 = \Delta_1 \cap Y_1$. 
On $E_1 \simeq \mathbb{P}^3$ with coordinates $(x,y,z,u)$, we see that 
\[
\Delta_1 \cap E_1 =(u=0) \subset E_1,  
\] 
\[
Y_1 \cap E_1= (g^{(2)}(x,y,z) + u h^{(1)}(x,y,z,u) =0)  \subset E_1, 
\] 
where $g^{(2)}(x,y,z)$ is the degree $2$ part of $g$ and $h^{(1)}$ is the degree $1$ part of $h$.     
By this description, we see that $\Delta_1 \cap E_1$ and $Y_1 \cap E_1$ have no common component. 
Thus we see that $D_1 = \Delta_1 \cap Y_1$ and it is a Cartier divisor. 

If $Y_1$ is smooth, we see that $D_2$ is Cartier. 
If $Y_1$ is singular, by the same argument, we see that $Y_2 \cap E_2$ and $\Delta_2 \cap E_2$ have no common component and 
$D_2 = Y_2 \cap \Delta_2$ 
since we can take local equations of $Y_2 \subset Z_2$ and $D_2 \subset Y_2$ as in (\ref{cDVequationform}) at a Du Val point $p_1 \in D_1$.   

We can proceed as this and show the claim for all $i$. 

\end{proof}

We can assume that $f_Y$ and $f_D$ are $G$-equivariant since we can take $G$-invariant centers of the blow-ups for $f_D$.  

Next, we can take a $G$-equivariant resolution $f_2 \colon \tilde{Y} \rightarrow Y_l$ such that 
$f_2$ is isomorphism on $Y_l \setminus \Sing Y_l$. 
Note that $f_2$ induces an isomorphism on $D_l$ since it is a smooth Cartier divisor on $Y_l$ and thus $Y_l$ is smooth around $D_l$. 
We see that the composition $f := f_Y \circ f_2 \colon \tilde{Y} \rightarrow Y$ satisfies the required condition. 
Thus we finish the proof of Proposition \ref{essresolprop}. 
\end{proof}

\subsection{Classification of $3$-fold terminal singularities}\label{terminalsect}
Let $(p \in U )$ be a germ of a $3$-fold terminal singularity. 
By Reid's result \cite{YPG}, $(U,p)$ is locally isomorphic to 
\[
0 \in (f =0) / \mathbb{Z}_r \subset \mathbb{C}^4/\mathbb{Z}_r,  
\]
where $\mathbb{Z}_r$ acts on $\mathbb{C}^4$ diagonally and 
$f \in \mathbb{C}[x,y,z,u]$ and $x,y,z,u$ are $\mathbb{Z}_r$-semi-invariant functions on $\mathbb{C}^4$. 
By the list in \cite{YPG}(6.4),  we have a $\mathbb{Z}_r$-semi-invariant function $h \in \mathbb{C}[x,y,z,u]$ such that 
\[
D_h :=(f= h = 0)/ \mathbb{Z}_r \subset (f=0)/\mathbb{Z}_r =:U_f
\]
 has only a Du Val singularity at the origin and 
$D_h \in |{-}K_{U_f}|$.

\subsection{Some ingredients for the proof}\label{ingsect}

Let $X$ be an algebraic scheme and $D$ its closed subscheme. 
For the functor $\Def_{(X,D)} \colon \mathcal{A} \rightarrow (Sets)$, let  
$T^1_{(X,D)} := \Def_{(X,D)}(A_1)$ be the tangent space.

We use the following fact that deformations of a pair of a variety and its divisor. 

\begin{lem}\label{codim3pair}
Let $X$ be a $3$-fold with only terminal singularities and $D$ a $\mathbb{Q}$-Cartier divisor on $X$. 
Let $Z \subset X$ be a $0$-dimensional subset. Let $\iota \colon U:= X \setminus Z \hookrightarrow X$ be an open immersion.
Set $D_U := D \cap U$.  
 
Then the restriction homomorphism $\iota^* \colon T^1_{(X,D)} \rightarrow T^1_{(U,D_U)}$ is an isomorphism.  
\end{lem}

\begin{proof}
We can construct the inverse $\iota_* \colon T^1_{(U,D_U)} \rightarrow T^1_{(X,D)}$ of $\iota^*$ as follows. 
$\xi \in T^1_{(U,D_U)}$ corresponds to a deformation $U_1 \rightarrow \Spec A_1$ and an $A_1$-flat ideal sheaf $\mathcal{I}_{D_{U_1}}$. 
We see that $\mathcal{O}_{X_1} := \iota_* \mathcal{O}_{U_1}$ is a sheaf of $A_1$-flat algebras by a similar argument as in the proof of Proposition \ref{obs}. 
Moreover, we see that $\mathcal{I}_{D_1}:= \iota_* \mathcal{I}_{D_{U_1}}$ is an $A_1$-flat ideal sheaf. 
Indeed there is an exact sequence $0 \rightarrow \mathcal{I}_{D_U} \rightarrow \mathcal{I}_{D_{U_1}} \rightarrow \mathcal{I}_{D_U} \rightarrow 0$ and, 
by taking its push-forward by $\iota$, we obtain an exact sequence
\begin{equation}\label{idealexact}
0 \rightarrow \mathcal{I}_D \rightarrow \mathcal{I}_{D_1} \rightarrow \mathcal{I}_D \rightarrow 0. 
\end{equation} 
The surjectivity in (\ref{idealexact}) follows from $R^1 \iota_* \mathcal{I}_{D_{U}} =0$. 
We can show that $R^1 \iota_* \mathcal{I}_{D_{U}} =0$ similarly as Claim \ref{omegaR1} since 
$\mathcal{I}_D$ can be written locally as an eigenspace of some invertible sheaf with respect to the group action induced by the 
index one cover.   
By the sequence (\ref{idealexact}), we see that $\mathcal{I}_{D_1}$ is flat over $A_1$. 
Consider the diagram
\[
\xymatrix{
\mathcal{I}_{D_1} \otimes_{A_1} (t) \ar[r]^{u_1} \ar[d]^{\alpha_1} & \mathcal{O}_{X_1} \otimes_{A_1} (t) \ar[d]^{\alpha_{2}} \\
\mathcal{I}_{D_1} \ar[r]^{u_2} & \mathcal{O}_{X_1}. 
}
\]
We see that $\alpha_1$ is injective since $\mathcal{I}_{D_1}$ is flat over $A_1$. 
Since $u_2$ is also injective, we see that $u_1$ is injective. 
Since $(t) \simeq \mathbb{C}$, we see that $\mathcal{I}_{D_1} \otimes_{A_1} \mathbb{C} \rightarrow \mathcal{O}_{X_1} \otimes_{A_1} \mathbb{C}$ 
is injective. 
By \cite[Corollary A.6]{Sernesi}, this implies that $\mathcal{O}_{D_1}= {\rm Coker} (u_2)$ is flat over $A_1$. 
Thus $(\mathcal{O}_{X_1}, \mathcal{I}_{D_1})$ defines an element $\iota_*(\eta) \in T^1_{(X,D)}$ and 
this determines $\iota_*$. 
\end{proof}

\vspace{5mm}

Let $p \in U$ be a Stein neighborhood of a $3$-fold terminal singularity $p$ with the Gorenstein index $r$. 
By the classification of $3$-fold terminal singularities, there exists $D \in |{-}K_U|$ with only Du Val singularity at $p$. 
Let $m$ be a positive multiple of $r$ and $\pi_U \colon V \rightarrow U$ the $\mathbb{Z}_m$-cyclic cover of $U$ determined by the isomorphism $\mathcal{O}_U(rK_U) \simeq \mathcal{O}_U$ 
as in Section \ref{useful}.  Set $\Delta:= \pi_U^{-1}(D)$. 
Then $V$ has  terminal Gorenstein singularities at $Q:= \pi^{-1}(p)$ and $\Delta$ has Du Val singularities at $Q$. 
Let $\nu \colon \tilde{V} \rightarrow V$ be the $\mathbb{Z}_m$-equivariant resolution of singularities of $(V, \Delta)$ constructed in Proposition \ref{essresolprop}. 
Let $\tilde{\Delta}:= \nu^{-1}_* (\Delta) \subset \tilde{V}$ be the strict transform of $\Delta$ and $F$ the exceptional divisor of $\nu$.  Then we have the coboundary map 
\begin{equation}\label{tauvdelta}
\tau_{(V,\Delta)} \colon H^1(V', \Omega^2_{V'}(\log \Delta')) \rightarrow H^2_F (\tilde{V}, \Omega^2_{\tilde{V}}(\log \tilde{\Delta})),  
\end{equation}
where $V' := V \setminus Q$ and $\Delta' := V' \cap \Delta$. 
By Lemma \ref{codim3pair}, we see that 
\[
T^1_{(V, \Delta)} \simeq T^1_{(V', \Delta')} \simeq H^1(V', \Theta_{V'}(- \log \Delta')) \simeq H^1( V', \Omega_{V'}^2(\log \Delta') (-K_{V'} - \Delta')).   
\]
 By fixing a $\mathbb{Z}_m$-equivariant isomorphism $\mathcal{O}_{V} \simeq \mathcal{O}_V (-K_V -\Delta)$, we finally obtain an isomorphism 
\[
T^1_{(V, \Delta)} \simeq H^1(V', \Omega^2_{V'}(\log \Delta')).
\] 
This isomorphism is $\mathbb{Z}_m$-equivariant and the $\mathbb{Z}_m$-invariant parts are
\[
T^1_{(U,D)} \simeq H^1(U', \Omega^2_{U'}(\log D')). 
\]
For deformations of $\Delta$, we have the following. 

\begin{lem}
Let $ \iota_{\Delta} \colon \Delta' \hookrightarrow \Delta$ be the open immersion. 
Then the restriction homomorphism $\iota_{\Delta}^* \colon T^1_{\Delta} \rightarrow T^1_{\Delta'}$ is injective. 
\end{lem}

\begin{proof}
For $\Delta_1 \in \Def_{\Delta}(A_1)$, we have $(\iota_{\Delta})_* \iota_{\Delta}^* \mathcal{O}_{\Delta_1} \simeq \mathcal{O}_{\Delta_1}$ since $\Delta$ is $S_2$. 
\end{proof}

We have the following commutative diagram; 
\[
\xymatrix{
T^1_{(V,\Delta)} \ar[r] \ar[d]^{p_{\Delta}} & T^1_{(V', \Delta')} \ar[r] \ar[d]^{p_{\Delta'}} & H^1(V', \Omega^2_{V'}(\log \Delta')) \ar[d]^{P_{\Delta'}} \\
T^1_{\Delta} \ar[r] & T^1_{\Delta'} \ar[r] & H^1(\Delta', \Omega^1_{\Delta'}),  
}
\]
where $P_{\Delta'}$ is induced by the residue homomorphism. 
This implies that the elements of $\Image P_{\Delta'}$ is coming from elements of $T^1_{\Delta}$. 
We also have the following diagram; 
\[
\xymatrix{
H^1(\tilde{\Delta}, \Omega^1_{\tilde{\Delta}}) \ar[r]^{R_{\Delta}} & H^1( \Delta', \Omega^1_{\Delta'}) \\
T^1_{\tilde{\Delta}} \ar[u]^{\simeq} \ar[r] \ar[rd]_{(\nu_{\Delta})_*} & T^1_{\Delta'} \ar[u]^{\simeq} \\
 & T^1_{\Delta} \ar[u]_{\iota_{\Delta}^*} 
 }
\]
The vertical isomorphisms are induced by the isomorphism $\mathcal{O}_{\Delta}(K_{\Delta}) \simeq \mathcal{O}_{\Delta}$ 
since $\nu_{\Delta}^* K_{\Delta} = K_{\tilde{\Delta}}$. 
The homomorphism $(\nu_{\Delta})_*$ is the blow-down morphism by Wahl (\cite{wahl}). 
It is well known that $(\nu_{\Delta})_* =0$ since $\Delta$ has a Du Val singularity (cf. \cite[2.10]{BW}). 
Hence we see that $R_{\Delta} =0$ as well.

We have the following lemma. 

\begin{lem}\label{rest0}
Let $R_{\Delta} \colon H^1(\tilde{\Delta}, \Omega^1_{\tilde{\Delta}}) \rightarrow H^1( \Delta', \Omega^1_{\Delta'})$ 
be the restriction homomorphism as above. 

Then we have $P_{\Delta'} (\Ker \tau_{(V, \Delta)}) \subset \Image R_{\Delta} =0$. 
In particular, if $\eta \in H^1(V', \Omega^2_{V'}(\log \Delta')) \simeq T^1_{(V,\Delta)}$ is a 
smoothing direction, then $\tau_{(V,\Delta)}(\eta) \neq 0$.  
\end{lem}

\begin{proof}
We have a diagram 
\[
\xymatrix{
H^1(\tilde{V}, \Omega^2_{\tilde{V}}(\log \tilde{\Delta})) \ar[r]^{\alpha_{(V, \Delta)}} \ar[d] & H^1(V', \Omega^1_{V'}(\log \Delta')) \ar[d] \\
H^1(\tilde{\Delta}, \Omega^1_{\tilde{\Delta}}) \ar[r]^{R_{\Delta}} & H^1( \Delta', \Omega^1_{\Delta'}) 
}
\]
The vertical homomorphisms are induced by the residue homomorphisms and the horizontal homomorphisms are induced by open immersions. 
Hence the diagram is commutative. 
Since $\Ker \tau_{(V, \Delta)} = \Image \alpha_{(V, \Delta)}$, we obtain the claim by the diagram. 
\end{proof}

We also need the following Lefschetz type statement. 

\begin{prop}\label{lefprop}
Let $Y$ be a normal projective $3$-fold with only isolated singularities and $\Delta \subset Y$ its ample Cartier divisor 
with only isolated singularities. 
Assume that $H^1(Y, \mathcal{O}_Y) =0$. Let $\nu \colon \tilde{Y} \rightarrow Y$ be a resolution of singularities of the pair $(Y, \Delta)$ which is isomorphism on $Y \setminus (\Sing Y \cup \Sing \Delta)$ such that 
  the strict transform $\tilde{\Delta}$ of $\Delta$ is smooth. Let $r_{\tilde{\Delta}} \colon \Pic \tilde{Y} \rightarrow \Pic \tilde{\Delta}$ be the restriction homomorphism. 

Then $\Ker r_{\tilde{\Delta}} $ is generated by $\nu$-exceptional divisors. 
\end{prop}

\begin{proof}
It is enough to show that 
\[
r_{\Delta} \colon \Cl Y \rightarrow \Cl \Delta
\]
is injective. Indeed we have a commutative diagram 
\[
\xymatrix{
\Cl \tilde{Y} \ar[r]^{r_{\tilde{\Delta}}} \ar[d]_{\nu_*} & \Cl \tilde{\Delta} \ar[d]^{(\nu_{\Delta})_*} \\
\Cl Y \ar[r]^{r_{\Delta}} & \Cl \Delta
}
\] 
and, if $r_{\Delta}$ is injective, can see that  
\[
\Ker r_{\tilde{\Delta}} \subset \Ker (\nu_{\Delta})_* \circ r_{\tilde{\Delta}} = 
\Ker r_{\Delta} \circ \nu_* = \Ker \nu_*  
\]
and $\Ker \nu_*$  is generated by $\nu$-exceptional divisors.

Let $m$ be a sufficiently large integer such that $m \Delta$ is very ample. 
By \cite[Theorem 1]{RSNL}, there exists a very general smooth element $\Delta_m \in | m \Delta|$ which is disjoint with $\Sing \Delta$ and 
\[
r_{\Delta_m} \colon \Cl Y \rightarrow \Cl \Delta_m 
\]
is an isomorphism. 
Take $A \in \Ker r_{\Delta}$. Then we have $A \cdot \Delta = 0$ as a rational equivalence class of a cycle on $Y$. 
Then we have 
\[
A \cdot \Delta_m =0 
\]
as a rational equivalence class on $Y$. 

We show that $A|_{\Delta_m}=0 \in \Cl \Delta_m$ as follows. 
It is enough to show that $A|_{\Delta_m}$ is numerically trivial on $\Delta_m$ since 
$H^1(\Delta_m, \mathcal{O}_{\Delta_m})=0$. 
Let $\Gamma \in \Cl \Delta_m$ be any element. Since $r_{\Delta_m}$ is an isomorphism, 
there exists $F \in \Cl Y $ such that $F|_{\Delta_m} = \Gamma$. 
We have 
\[
A|_{\Delta_m} \cdot \Gamma = (A \cdot \Delta_m)\cdot F =0 
\] 
by the intersection theory. 
Indeed $A \cdot \Delta_m$ is a sum of several curves which are regularly immersed since 
$\Delta_m \cap \Sing Y = \emptyset$. 
Hence $A|_{\Delta_m} = 0 \in \Cl \Delta_m$ and we get $A = 0 \in \Cl Y$ since $\Cl Y \stackrel{\simeq}{\rightarrow} \Cl \Delta_m$. 
Thus we get $r_{\Delta}$ is injective and we finish the proof.    
\end{proof}

\subsection{Proof of Theorem \ref{simulqsmqfanointro}}\label{simulqsmproofsection}

Our strategy of the proof of Theorem \ref{simulqsmqfanointro} is similar to that of \cite[Theorem 1.3]{NamSt}. 
In \cite[Theorem 1.3]{NamSt}, there are two crucial ingredients. 
One is the non-vanishing of the coboundary map of local cohomology group (\cite[Theorem 1.1]{NamSt}). 
And another is the vanishing of a composition of homomorphisms between some cohomology groups (\cite[Proposition 1.2]{NamSt}). 
We modify these propositions to our setting of a pair of a variety and its divisor.

\begin{proof}[Proof of Theorem \ref{simulqsmqfanointro}] 
By Corollary \ref{qsmordqfano}, we can assume that the singularities on $X$ are non ordinary terminal singularities. 
Since the forgetful morphism $\Def_{(X,D)} \rightarrow \Def_X$ is smooth  by Theorem \ref{pairunobs}, 
we see that $D \in |{-}K_X|$ extends sideways in a deformation of $X$.  
We prepare the notations to introduce the diagram (\ref{logdiag}).

Let $m$ be a positive integer such that $-m K_X$ is very ample and  $|{-}mK_X|$ contains a smooth element $D_m$ which satisfies 
$D_m \cap \Sing D   = \emptyset$ and intersects transversely with $D$. 
Let $\pi \colon Y := \Spec \oplus_{i=0}^{m-1} \mathcal{O}_X (i K_X) \rightarrow X$ be the cyclic cover determined by $D_m$. 
Note that $Y$ is terminal Gorenstein. 
Put $\{ p_1,\ldots, p_l \} := \Sing D$. 
Note that $\Sing X \subset \Sing D$ since all the singularities on $X$ are non-Gorenstein. Also note that $G := \Gal(Y/X) \simeq \mathbb{Z}_m$ acts on $Y$ and $\Delta := \pi^{-1}(D)$ is 
$G$-invariant.

Let $U_i$ be a sufficiently small Stein neighborhood of $p_i$ such that $U_i \setminus \{p_i \}$ is smooth and $K_{V_i} =0$, where $V_i := \pi^{-1}(U_i)$. 
Let $\pi_i : V_i \rightarrow U_i$ be the morphism induced by $\pi$. 

By Proposition \ref{essresol}, we can take a $\mathbb{Z}_m$-equivariant resolution 
$\nu \colon \tilde{Y} \rightarrow Y$ of $Y$ such that $\nu|_{\nu^{-1}(Y \setminus \Sing \Delta)}$ is an isomorphism, $\tilde{\Delta} := (\nu^{-1})_* \Delta$ is smooth and  
\[
\nu_{\Delta}^* K_{\Delta} = K_{\tilde{\Delta}}, 
\] 
where $\nu_{\Delta} \colon \tilde{\Delta} \rightarrow \Delta$ is induced by  $\nu$.  
Then we have the following diagram; 
\begin{equation}\label{globaldiag2}
\xymatrix{
      \tilde{Y} \ar[r]^{\tilde{\pi}} \ar[d]^{\nu} & \tilde{X} \ar[d]^{\mu} \\
     Y \ar[r]^{\pi}& X .
     }
 \end{equation}
We also have the following diagram induced by the above diagram;  
\begin{equation}\label{localdiag2}
\xymatrix{
      \tilde{V}_{i} \ar[r]^{\tilde{\pi}_{i}} \ar[d]^{\nu_{i}} & \tilde{U}_i \ar[d]^{\mu_i} \\
      V_{i} \ar[r]^{\pi_{i}} & U_i .
     }
 \end{equation}
Put $F := \Exc (\nu)\subset \tilde{Y}$, $F_i := \Exc(\nu_i)$, $E:= \Exc (\mu)$ and $E_i:= \Exc (\mu_i)$.  
Put $ \tilde{\Delta}_i := (\nu_i^{-1})_* \Delta_i$, where $\Delta_i := \Delta \cap V_i$.

Let $\mathcal{F}^{(0)}$ be the $\mathbb{Z}_m$-invariant part of 
$\tilde{\pi}_* \Omega^2_{\tilde{Y}}(\log \tilde{\Delta})$ and set $\mathcal{F}_i^{(0)}:= \mathcal{F}^{(0)}|_{\tilde{U}_i} $.    
Set $U: = X \setminus \Sing D$. 
Note that $\mathcal{F}^{(0)}|_U \simeq \Omega^2_U(\log D_U)$, where $D_U:= D \cap U$. 

Hence we have the following diagram;  
\begin{equation}\label{logdiag}
\xymatrix{
H^1(U, \Omega^2_U(\log D_U) ) \ar[r]^{\oplus \psi_i} \ar[d]^{\oplus p_{U_i}} & \bigoplus_{i=1}^l H^2_{E_i}(\tilde{X}, \mathcal{F}^{(0)} ) \ar[d]^{\simeq} \ar[r]^{\oplus \beta_i}
 &     H^2(\tilde{X}, \mathcal{F}^{(0)} ) \\
     \bigoplus_{i=1}^l H^1(U'_i, \Omega^2_{U'_i}(\log D'_i)) \ar[r]^{\oplus \phi_i}&  \bigoplus_{i=1}^l H^2_{E_i}( \tilde{U}_i, \mathcal{F}_i^{(0)}),   &  
}
 \end{equation}
 where $U'_i:= U_i \setminus \{ p_i\}$ and  $D'_i := D \cap U'_i$.

We have restriction homomorphisms $\iota^* \colon T^1_{(X,D)} \rightarrow T^1_{(U,D_U)}$ and 
$\iota_i^* \colon T^1_{(U_i,D_i)} \rightarrow T^1_{(U'_i, D'_i)}$, where $\iota \colon U \hookrightarrow X$ and $\iota_i \colon U'_i \hookrightarrow U_i$ 
are open immersions.  
By Lemma \ref{codim3pair} and the arguments around it, 
we see that 
\[
H^1(U, \Omega^2_U(\log D_U) ) \simeq T^1_{(X,D)},  
\]
\[
 H^1(U'_i, \Omega^2_{U'_i}(\log D'_i)) \simeq T^1_{(U_i,D_i)}. 
\]
 By using the diagram (\ref{logdiag}), we want to lift $\eta_i \in H^1(U'_i, \Omega^2_{U'_i}(\log D'_i)) \simeq T^1_{(U_i, D_i)}$ 
 which induces a simultaneous $\mathbb{Q}$-smoothing of $(U_i,D_i)$ to $X$. 
 For that purpose, we consider $\phi_i(\eta_i)$ and lift it to $H^1(U, \Omega_U^2(\log D_U))$.

Note that $\phi_i$ is the $\mathbb{Z}_m$-invariant part of the coboundary map
$
\tau_i \colon H^1(V'_i, \Omega_{V'_i}^2(\log \Delta'_i)) 
\rightarrow H^2_{F_i}(\tilde{V}_i, \Omega^2_{\tilde{V}_i}( \log \tilde{\Delta}_i)).  
$ We see that $\tau_i$ is same as $\tau_{(V_i, \Delta_i)}$ introduced in (\ref{tauvdelta}). 
Thus we can use the results in Section \ref{ingsect}. 
By Lemma \ref{rest0}, we see that 
\begin{equation}\label{logkerim}
P_{\Delta'_i} (\Ker \tau_i) \subset \Image R_{\Delta_i} =0,
\end{equation} where 
$P_{\Delta'_i} \colon H^1(V_i', \Omega^2_{V'_i}(\log \Delta'_i)) \rightarrow H^1(\Delta'_i, \Omega^1_{\Delta'_i})$ and 
$R_{\Delta_i} \colon H^1( \tilde{\Delta}_i, \Omega^1_{\tilde{\Delta}_i}) \rightarrow H^1( \Delta'_i, \Omega^1_{\Delta'_i}) $ are defined as in Section \ref{ingsect}. 

There exists $\eta_i \in T^1_{(U_i,D_{i})}$ which induces a simultaneous $\mathbb{Q}$-smoothing of $(U_i,D_i)$ 
by the description in Section \ref{terminalsect}.   
Note that $\phi_i( \eta_i) \not= 0$ by the relation (\ref{logkerim}). 
To lift $\phi_i(\eta_i)$ to $H^1(U, \Omega_U^2(\log D_U))$, we need the following claim.

\begin{claim}\label{comp0}
$\beta_i \circ \phi_i =0$. 
\end{claim}

\begin{proof}[Proof of Claim]
$\beta_i \circ \phi_i$ is the $\mathbb{Z}_m$-invariant part of a composition of the homomorphisms 
\begin{multline}
H^1(V'_i, \Omega^2_{V'_i}(\log \Delta'_i)) \rightarrow H^2_{F_i}(\tilde{V}_i, 
\Omega^2_{\tilde{V}_i}(\log \tilde{\Delta}_i)) \\ 
\simeq H^2_{F_i}(\tilde{Y}, \Omega^2_{\tilde{Y}}(\log \tilde{\Delta})) 
\rightarrow H^2(\tilde{Y}, \Omega^2_{\tilde{Y}}(\log \tilde{\Delta})). 
\end{multline}
By considering its dual, it is enough to show that the $\mathbb{Z}_m$-invariant part of the homomorphism 
\[
\Phi_i \colon H^1(\tilde{Y}, \Omega^1_{\tilde{Y}}(\log \tilde{\Delta})(-\tilde{\Delta})) 
\rightarrow H^1(V'_i, \Omega^1_{V'_i}(\log \Delta'_i)(-\Delta'_i)) 
\]
is zero. We show that $\Phi_i =0$ in the following. 

For a $\mathbb{Z}$-module $M$, we set $M_{\mathbb{C}}:= M \otimes \mathbb{C}$.  
Let $\mathcal{K}_{(\tilde{Y},\tilde{\Delta})}$ be a sheaf of groups defined by an exact sequence 
\[
1 \rightarrow \mathcal{K}_{(\tilde{Y},\tilde{\Delta})} \rightarrow \mathcal{O}^*_{\tilde{Y}} \rightarrow 
 \mathcal{O}_{\tilde{\Delta}}^* \rightarrow 1. 
\]
We have a commutative diagram with two horizontal exact sequences 
\[
\xymatrix{ 0 \ar[r] & H^1(\tilde{Y}, \Omega^1_{\tilde{Y}}(\log \tilde{\Delta})(- \tilde{\Delta})) \ar[r] & H^1(\tilde{Y}, \Omega^1_{\tilde{Y}}) \ar[r] & 
H^1(\tilde{\Delta}, \Omega^1_{\tilde{\Delta}}) \\ 
0 \ar[r] & H^1(\tilde{Y}, \mathcal{K}_{(\tilde{Y}, \tilde{\Delta})})_{\mathbb{C}} \ar[r] \ar[u]^{\epsilon} & 
H^1(\tilde{Y}, \mathcal{O}_{\tilde{Y}}^*)_{\mathbb{C}} \ar[r] \ar[u]^{\delta_{\tilde{Y}}} & H^1(\tilde{\Delta}, \mathcal{O}_{\tilde{\Delta}}^*)_{\mathbb{C}} 
\ar[u]^{\delta_{\tilde{\Delta}}}, 
}  
\]
where the injectivity follows since  we see that  $H^0(\tilde{\Delta}, \Omega^1_{\tilde{\Delta}}) =0$ and   
that  $H^0(\tilde{Y}, \mathcal{O}_{\tilde{Y}}^*) \rightarrow H^0(\tilde{\Delta},\mathcal{O}_{\tilde{\Delta}}^*)$ is surjective. 
We see that $\delta_{\tilde{Y}}$ is an isomorphism and $\delta_{\tilde{\Delta}}$ is injective since 
we have $H^i(\tilde{Y}, \mathcal{O}_{\tilde{Y}})=0$ for $i=1,2$ and $H^1(\tilde{\Delta}, \mathcal{O}_{\tilde{\Delta}})=0$. 
Hence we see that $\epsilon$ is an isomorphism. 

Set $\mathcal{K}_{(V'_i, \Delta'_i)}:= \mathcal{K}_{(\tilde{Y}, \tilde{\Delta})}|_{V'_i}$. We have a commutative diagram 
\[
\xymatrix{
H^1(\tilde{Y}, \Omega^1_{\tilde{Y}}(\log \tilde{\Delta})(-\tilde{\Delta}))  \ar[d] 
&  H^1(\tilde{Y}, \mathcal{K}_{(\tilde{Y}, \tilde{\Delta})})_{\mathbb{C}} \ar[l]^{\simeq} \ar[d]^{\Phi'_i} \\ 
H^1(V'_i, \Omega^1_{V'_i}(\log \Delta'_i)(-\Delta'_i))  & H^1(V'_i, \mathcal{K}_{(V'_i, \Delta'_i)})_{\mathbb{C}}. \ar[l] 
}
\]
Hence it is enough to show that $\Phi'_i=0$. 
Moreover we have a commutative diagram 
\[
\xymatrix{
H^1(\tilde{Y}, \mathcal{K}_{(\tilde{Y}, \tilde{\Delta})})_{\mathbb{C}} \ar[d] \ar@{^{(}->}[r] & H^1(\tilde{Y}, \mathcal{O}_{\tilde{Y}}^*)_{\mathbb{C}} \ar[d]^{\Phi_i''} \\
H^1(V'_i, \mathcal{K}_{(V'_i, \Delta'_i)})_{\mathbb{C}} \ar@{^{(}->}[r] & H^1(V'_i, \mathcal{O}_{V'_i}^*)_{\mathbb{C}}.
}
\]
Since $\nu$ is an isomorphism outside $\Sing \Delta$, we see that $\Phi_i'' =0$ by Proposition \ref{lefprop}. 
Hence we see that $\Phi'_i =0$ and we finish the proof of Claim \ref{comp0}. 
\end{proof}

\vspace{5mm}

By Claim \ref{comp0}, we have $\beta_i (\phi_i(\eta_i)) =0$. 
Thus there exists $\eta \in H^1(U, \Omega^2_U( \log D'))$ 
such that $\psi_i(\eta) = \phi_i(\eta_i)$ for each $i$. 
Then $P_{\Delta'_i}(p_{U_i} (\eta) - \eta_i)  
\in P_{\Delta'_i}(\Ker \tau_i) \subset 
\Image R_{\Delta_i} =0$ by the relation (\ref{logkerim}). 
Hence we have 

\begin{equation}\label{equaleta}
P_{\Delta'_i}(p_{U_i}(\eta)) = P_{\Delta'_i}(\eta_i) \in H^1(\Delta'_i, \Omega^1_{\Delta'_i}).
\end{equation} 
Note that this element corresponds to an element of $T^1_{\Delta_i}$ which induces a smoothing of $\Delta_i$ 
by the definition of $\eta_i$. 

By Theorem \ref{pairunobs}, there exists a deformation $f \colon (\mathcal{X}, \mathcal{D}) \rightarrow \Delta^1$ of $(X,D)$ induced by $\eta$.  
By the relation (\ref{equaleta}), we see that $f$ induces a smoothing of $\Delta_i$. 
Note that $\Sing V_i \subset \Sing \Delta_i$ and this relation is preserved by deformation 
since $\mathcal{D}_t \in |{-}K_{X_t}|$ contains all non-Gorenstein points of $X_t$, where $X_t := f^{-1}(t)$ for $t \in \Delta^1$. 
We see that a deformation of $V_i$ becomes smooth along a deformation of $\Delta_i$ which is smooth since a deformation of $\Delta_i \subset V_i$ is still a Cartier divisor.  
Thus $f$ is a $\mathbb{Q}$-smoothing and we finish the proof of Theorem \ref{simulqsmqfanointro}.   
\end{proof}

\begin{eg}
We give an example of a $\mathbb{Q}$-Fano 3-fold $X$ such that $|{-}K_X|$ does not contain a Du Val elephant (\cite[4.8.3]{ABR}). 

Let $S_{14} \subset \mathbb{P}(2,2,3,7)$ be the surface defined by a polynomial $w^2= y_1^3 y_2^4 -y_1 z^4$. 
Then $S_{14}$ has an elliptic singularity at $[0:1:0:0]$. 
Let $X_{14} \subset \mathbb{P}(1,2,2,3,7)$ be suitable extension of $S_{14}$ by adding several terms including $x$. 
Then we see that $X_{14}$ is terminal and $|{-}K_X| = \{S_{14} \}$ with non Du-Val singularity. 
This $(X,D)$ admits a simultaneous $\mathbb{Q}$-smoothing since $X$ is a quasismooth well-formed weighted hypersurface.   
\end{eg}

\subsection{Genus bound for primary $\mathbb{Q}$-Fano 3-folds} 

\begin{defn}
Let $X$ be a $\mathbb{Q}$-Fano 3-fold. Let $\tilde{\Cl} X$ be the quotient of  the divisor class group $\Cl X$ by its torsion part. 
$X$ is called {\it primary} if  
\[
\tilde{\Cl} X \simeq \mathbb{Z} \cdot [-K_X ] .
\]
\end{defn}

Takagi \cite{Takdv} proved the following theorem on the genus bound of certain primary $\mathbb{Q}$-Fano 3-folds.

\begin{thm}\label{takagi} $($\cite[Theorem 1.5]{Takdv}$)$
Let $X$ be a primary $\mathbb{Q}$-Fano $3$-fold with only terminal quotient singularities. 
Assume that $X$ is non-Gorenstein and $|{-}K_X|$ contains an element 
with only Du Val singularities. 

Then $h^0(X, -K_X) \le 10$. 
\end{thm}

By combining his result and our results, we get the following genus bound. 

\begin{thm}\label{genus}
Let $X$ be a primary $\mathbb{Q}$-Fano $3$-fold. 
Assume that $X$ is non-Gorenstein and $|{-}K_X|$ contains an element 
with only Du Val singularities. 

Then $h^0(X,-K_X) \le 10$. 
\end{thm}
 
\begin{proof}
By Theorem \ref{simulqsmqfanointro}, there is a deformation $\mathcal{X} \rightarrow \Delta^1$ of $X$ such that $\mathcal{X}_t$ has only quotient singularities and $|{-} K_{\mathcal{X}_t}|$ contains an element with only Du Val singularities for $t  \neq 0$. 
By Theorem 5.28 of \cite{KM}, we have  $h^0(X, -K_X) = h^0(\mathcal{X}_t, -K_{\mathcal{X}_t})$. 
By Theorem \ref{takagi}, we have 
\[ 
h^0(X, -K_X) = h^0(\mathcal{X}_t, -K_{\mathcal{X}_t}) \le 10. 
\]
\end{proof}

\section*{Acknowledgment}
The author is grateful to his advisor Professor Miles Reid for suggesting him the problem of $\mathbb{Q}$-smoothings,   
warm encouragements and valuable comments. 
He would like to thank Professor Yujiro Kawamata for warm encouragement and valuable comments.  
He thanks Professor Yoshinori Namikawa for answering his questions many times through e-mails, discussions and pointing out 
mistakes of the first draft. 
His comments about Schlessinger's result was also very useful for the proof of the unobstructedness. 
He thanks Professor Edoardo Sernesi for teaching him things around Proposition 2.4.8 \cite{Sernesi} and 
letting him know the paper \cite{FM}.  
He thanks Professor Tatsuhiro Minagawa for useful discussions. 
He thanks Professor Hiromichi Takagi for providing many informations through e-mails. 
He thanks Professor Shunsuke Takagi for answering his questions about local cohomology. 
He thanks Professor Osamu Fujino for checking a manuscript on Section \ref{essresol} carefully. 
He also thanks to Professors Donu Arapura, S\'{a}ndor Kov\'{a}cs and Karl Schwede 
for answering his questions concerning this problem through the webpage 
``MathOverflow''. 
He thanks Professor Yoshinori Gongyo and Doctor Tomoyuki Hisamoto for useful comments. 
He thanks Professor J\'{a}nos Koll\'{a}r for useful comments and letting him know 
about the references \cite{kollarflat}, \cite{kollarmoriflip}. 
He thanks Professor Vasudevan Srinivas for useful comments on Lefschetz injectivity of class groups. 
Finally, he would like to thank the referee for many constructive comments on the manuscript. 

He was partially supported by Warwick Postgraduate Research Scholarship. 
He was partially funded by Korean Government WCU
Grant R33-2008-000-10101-0, Research Institute for Mathematical Sciences and Higher School of Economics.

\end{document}